\documentclass[a4paper,10pt]{amsart}
\usepackage[top=0.8in, bottom=0.8in, left=1.25in, right=1.25in]{geometry}
\usepackage{float}
\usepackage{fancybox}
\usepackage{amsmath}
\usepackage{amssymb}
\usepackage{amsthm}
\usepackage{mathrsfs}
\usepackage{epstopdf}
\usepackage{array}
\usepackage{amscd}
\usepackage{color}
\usepackage{hyperref}
\usepackage[colorinlistoftodos]{todonotes}
\usepackage{comment}
\usepackage[shortlabels]{enumitem}
\usepackage{soul}

\usepackage[all]{xypic}

\usepackage{amsfonts}
\usepackage{amssymb}
\usepackage{mathrsfs}

\DeclareFontFamily{U}{rsfs}{%
\skewchar\font127}
\DeclareFontShape{U}{rsfs}{m}{n}{%
<-6>rsfs5<6-8.5>rsfs7<8.5->rsfs10}{}
\DeclareSymbolFont{rsfs}{U}{rsfs}{m}{n}
\DeclareSymbolFontAlphabet
{\mathrsfs}{rsfs}
\DeclareRobustCommand*\rsfs{%
\@fontswitch\relax\mathrsfs}

\theoremstyle{plain}
\newtheorem{thm}{Theorem}[section]
\newtheorem{prop}[thm]{Proposition}

\newtheorem{cor}[thm]{Corollary}

\newtheorem{prop-defi}[thm]{Proposition-Definition}
\newtheorem{thm-defi}[thm]{Theorem-Definition}
\newtheorem{lem-defi}[thm]{Lemma-Definition}

\newtheorem{conj}[thm]{Conjecture}

\theoremstyle{definition}
\newtheorem{rmk}[thm]{Remark}
\newtheorem{exam}[thm]{Example}

\newdimen\argwidth
\def\db[#1\db]{
 \setbox0=\hbox{$#1$}\argwidth=\wd0
 \setbox0=\hbox{$\left[\box0\right]$}
  \advance\argwidth by -\wd0
 \left[\kern.3\argwidth\box0 \kern.3\argwidth\right]}

\newcommand{\cC}{\mathcal{C}}

\newcommand{\eE}{\mathcal{E}}

\newcommand{\hH}{\mathcal{H}}
\newcommand{\iI}{\mathcal{I}}

\newcommand{\lL}{\mathcal{L}}

\newcommand{\oO}{\mathcal{O}}

\newcommand{\zZ}{\mathcal{Z}}

\newcommand\mdot{{\scriptscriptstyle\bullet}}

\newcommand{\Hom}{\mathop{\rm Hom}\nolimits}

\newcommand{\pt}{\textrm{pt}}

\newcommand{\dR}{\mathbf{R}}
\newcommand{\dL}{\mathbf{L}}

\newcommand{\Hilb}{\mathop{\rm Hilb}\nolimits}

\newcommand{\Pic}{\mathop{\rm Pic}\nolimits}

\newcommand{\ch}{\mathop{\rm ch}\nolimits}
\newcommand{\rk}{\mathop{\rm rk}\nolimits}

\newcommand{\Ext}{\mathop{\rm Ext}\nolimits}
\newcommand{\Spec}{\mathop{\rm Spec}\nolimits}

\newcommand{\Coh}{\mathop{\rm Coh}\nolimits}

\newcommand{\cneq}{\mathrel{\raise.095ex\hbox{:}\mkern-4.2mu=}}
\newcommand{\eqcn}{\mathrel{=\mkern-4.5mu\raise.095ex\hbox{:}}}

\newcommand{\RHom}{\mathop{\dR\mathrm{Hom}}\nolimits}

\makeatletter
 
  \@addtoreset{equation}{section}
\makeatother

\title[Stable pair invariants of local Calabi-Yau 4-folds]
{Stable pair invariants of local Calabi-Yau 4-folds}
\author{Yalong Cao}
\address{Kavli Institute for the Physics and Mathematics of the Universe (WPI),The University of Tokyo Institutes for Advanced Study, The University of Tokyo, Kashiwa, Chiba 277-8583, Japan}
\email{yalong.cao@ipmu.jp}
\author{Martijn Kool}
\address{Mathematical Institute, Utrecht University, P.O.~Box 80010 3508 TA Utrecht, The Netherlands}
\email{m.kool1@uu.nl}
\author{Sergej Monavari}
\address{Mathematical Institute, Utrecht University, P.O.~Box 80010 3508 TA Utrecht, The Netherlands}
\email{s.monavari@uu.nl}
 
\begin{document}
\maketitle
\begin{abstract}
In 2008, Klemm-Pandharipande defined Gopakumar-Vafa type invariants of a Calabi-Yau 4-fold $X$ using Gromov-Witten theory.
Recently, Cao-Maulik-Toda proposed a conjectural description of these invariants in terms of stable pair theory.

When $X$ is the total space of the sum of two line bundles over a surface $S$, and all stable pairs are scheme theoretically supported on the zero section, we express stable pair invariants in terms of intersection numbers on Hilbert schemes of points on $S$.
As an application, we obtain new verifications of the Cao-Maulik-Toda conjectures for low degree curve classes and find connections to Carlsson-Okounkov numbers. 
Some of our verifications involve genus zero Gopakumar-Vafa type invariants recently determined in the context of the log-local principle by Bousseau-Brini-van Garrel.

Finally, using the vertex formalism, we provide a few more verifications of the Cao-Maulik-Toda conjectures when thickened curves contribute and also for the case of local $\mathbb{P}^3$.
\end{abstract}
%\tableofcontents

\section{Introduction}

\subsection{GW/GV invariants of Calabi-Yau 4-folds}

Gromov-Witten invariants are rational numbers, which are virtual counts of stable maps from curves to a fixed algebraic variety.
Due to multiple cover contributions, they are in general not integers. For Calabi-Yau 4-folds, Klemm-Pandharipande \cite{KP} defined 
Gopakumar-Vafa type invariants using Gromov-Witten theory and conjectured their integrality.
More specifically, let $X$ be a smooth projective Calabi-Yau 4-fold.
Gromov-Witten invariants
vanish for genus $g\geqslant2$ for dimensional reasons and one only needs to consider the genus zero and one cases.

The genus zero Gromov-Witten invariants of $X$ for class $\beta \in H_2(X,\mathbb{Z})$ are defined using
an insertion. Consider the evaluation map $\mathrm{ev} \colon \overline{M}_{0, 1}(X, \beta)\to X$. For $\gamma \in H^{4}(X, \mathbb{Z})$,
one defines
\begin{equation}
\mathrm{GW}_{0, \beta}(\gamma)
=\int_{[\overline{M}_{0, 1}(X, \beta)]^{\rm{vir}}}
\mathrm{ev}^{\ast}(\gamma).
\nonumber \end{equation}
The genus zero Gopakumar-Vafa type invariants 
\begin{align}\label{intro:n}
n_{0, \beta}(\gamma) \in \mathbb{Q}
\end{align}
are defined in \cite{KP} by the identity
\begin{align*}
\sum_{\beta>0}\mathrm{GW}_{0, \beta}(\gamma) \, q^{\beta}=
\sum_{\beta>0}n_{0, \beta}(\gamma) \sum_{d=1}^{\infty}
d^{-2}q^{d\beta},
\end{align*}
where the sum is over all non-zero effective classes in $H_2(X,\mathbb{Z})$. For the genus one case, the virtual dimension of $\overline{M}_{1, 0}(X, \beta)$ is zero and one defines
\begin{align*}
\mathrm{GW}_{1, \beta}=\int_{[\overline{M}_{1, 0}(X, \beta)]^{\rm{vir}}}
1 \in \mathbb{Q}.
\end{align*}
The genus one Gopakumar-Vafa type invariants
\begin{align}\label{intro:n1}
n_{1, \beta} \in \mathbb{Q}
\end{align}
 are defined in~\cite{KP} by the identity
\begin{align*}
\sum_{\beta>0}
\mathrm{GW}_{1, \beta} \, q^{\beta}=
&\sum_{\beta>0} n_{1, \beta} \sum_{d=1}^{\infty}
\frac{\sigma(d)}{d} \, q^{d\beta}
+\frac{1}{24}\sum_{\beta>0} n_{0, \beta}(c_2(X))\log(1-q^{\beta}) \\
&-\frac{1}{24}\sum_{\beta_1, \beta_2}m_{\beta_1, \beta_2}
\log(1-q^{\beta_1+\beta_2}),
\end{align*}
where $\sigma(d)=\sum_{i|d}i$ and $m_{\beta_1, \beta_2}\in\mathbb{Z}$ are called meeting invariants, which can be inductively determined by the genus zero Gromov-Witten invariants of $X$.
In~\cite{KP}, both of the invariants (\ref{intro:n}), (\ref{intro:n1})
are conjectured to be integers.
Using localization techniques and mirror symmetry, they calculate the Gromov-Witten invariants of $X$ in numerous examples in support of their integrality conjecture.
The genus zero integrality conjecture has been proved by Ionel-Parker using symplectic geometry \cite[Thm.~9.2]{IP}.

\subsection{Stable pair invariants of Calabi-Yau 4-folds}\label{sect on pair inv on cpt CY4}

Stable pairs were introduced in general by Le Potier \cite{LeP} and used by Pandharipande-Thomas to define virtual invariants of smooth projective threefolds \cite{PT1, PT2, PT3}. Stable pair invariants of threefolds are related to Gromov-Witten invariants by the celebrated GW/PT correspondence \cite{MNOP, PT1}, which has been proved in many cases by Pandharipande-Pixton \cite{PP1, PP2}.

In \cite{CMT2}, Cao-Maulik-Toda studied stable pair theory of a smooth projective Calabi-Yau 4-fold $X$. They used stable pair invariants of $X$ to give a sheaf theoretical interpretation of the Gopakumar-Vafa type invariants (\ref{intro:n}) and (\ref{intro:n1})\,\footnote{In 
\cite{CMT1, CT2}, the authors also proposed a sheaf theoretical interpretation of (\ref{intro:n}), (\ref{intro:n1}) using Donaldson-Thomas type counting invariants of one dimensional 
stable sheaves on $X$.}.

Let $P_n(X,\beta)$ be the moduli space of stable pairs $\{s:\oO_X\to F\}$ with $\ch(F)=(0,0,0,\beta,n)$. 
There exists a virtual class 
\begin{equation}\label{pair moduli vir class intro}
[P_n(X, \beta)]^{\rm{vir}}\in H_{2n}\big(P_n(X, \beta),\mathbb{Z}\big), 
\end{equation}
in the sense of Borisov-Joyce \cite{BJ}, which depends on the choice of an orientation of a certain (real) line bundle over $P_n(X, \beta)$ \cite{CGJ}. 
For $\gamma \in H^{4}(X, \mathbb{Z})$, we define primary insertions
\begin{align*}
\tau \colon H^{4}(X,\mathbb{Z})\to H^{2}(P_n(X,\beta),\mathbb{Z}), \quad
\tau(\gamma)=\pi_{P\ast}(\pi_X^{\ast}\gamma \cup\ch_3(\mathbb{F}) ),
\end{align*}
where $\pi_X$, $\pi_P$ are projections from $X \times P_n(X,\beta)$
to the corresponding factors and 
$$
\mathbb{I}^{\mdot}=\{\oO \to \mathbb{F}\}
$$ 
is the universal stable pair on $X \times P_n(X,\beta)$. Note that $\ch_3(\mathbb{F})$ is 
Poincar\'e dual to the fundamental cycle of $\mathbb{F}$. The stable pair invariants of $X$ with primary insertions are defined by 
\begin{align}\label{inv cpt case}
P_{n,\beta}(\gamma):=\int_{[P_n(X,\beta)]^{\rm{vir}}}\tau(\gamma)^n. 
\end{align}
When $n=0$, we simply denote this invariant by $P_{0,\beta}$. We set $P_{0,0}:=1$ and $n_{0,0}(\gamma):=0$.
\begin{conj}\emph{(\cite{CMT2})} \label{conj:GW/GV g=0}
Let $X$ be a smooth projective Calabi-Yau 4-fold, $\beta \in H_2(X,\mathbb{Z})$, $\gamma \in H^4(X,\mathbb{Z})$, and $n \geqslant 1$. Then there exist choices of orientations such that
\begin{align*}
P_{n,\beta}(\gamma)=\sum_{\begin{subarray}{c} \beta_0+\beta_1+\cdots+\beta_n=\beta  \\ \beta_0, \beta_1,\ldots,\beta_n\geqslant0 \end{subarray}}P_{0,\beta_0}\cdot\prod_{i=1}^nn_{0,\beta_i}(\gamma),  \end{align*}
where the sum is over all effective decompositions of $\beta$.
\end{conj}
\begin{conj}\emph{(\cite{CMT2})} \label{conj:GW/GV g=1}
Let $X$ be a smooth projective Calabi-Yau 4-fold. Then there exist choices of orientations such that
\begin{align*}
\sum_{\beta \geqslant 0}
P_{0, \beta} \, q^{\beta}=
\prod_{\beta>0} M\big(q ^{\beta}\big)^{n_{1, \beta}},
\end{align*}
where $M(q)=\prod_{k\geqslant 1}(1-q^{k})^{-k}$ denotes the MacMahon function.
\end{conj}  
Conjecture \ref{conj:GW/GV g=0} can be interpreted as a wall-crossing formula in the category of D0-D2-D8 bound states in Calabi-Yau 4-folds \cite{CT1},
while Conjecture \ref{conj:GW/GV g=1} seems to be more mysterious.
In \cite{CMT2}, these conjectures were verified in the following cases (modulo some minor assumptions in some of the cases). In each case, Conjecture \ref{conj:GW/GV g=0} was only verified for $n=1$.
\begin{itemize}
\item $X$ is a general sextic and $\beta = [\ell], 2[\ell]$, where $\ell \subseteq X$ is a line.
\item $X$ is a Weierstrass elliptic fibration and $\beta = r [F]$, where $[F]$ is the fibre class and $r>0$ (in the case of Conjecture \ref{conj:GW/GV g=0} only for $r=1$).
\item $X = Y \times E$, where $Y$ is a smooth projective Calabi-Yau threefold, $E$ is an elliptic curve, and $\beta$ is the push-forward of an irreducible class on $Y \times \{\mathrm{pt}\}$.
\item $X = Y \times E$, where $Y$ is a smooth projective Calabi-Yau threefold, $E$ is an elliptic curve, and $\beta = r [E]$, where $[E]$ is the fibre class and $r>0$ (Conjecture \ref{conj:GW/GV g=1} only).
\end{itemize}
When $X$ is either the total space of a smooth projective Fano threefold, or the total space of $\oO(-1) \oplus \oO(-2)$ on $\mathbb{P}^2$, or $\oO(-1,-1) \oplus \oO(-1,-1)$ on $\mathbb{P}^1 \times \mathbb{P}^1$, the moduli spaces $P_n(X,\beta)$ are projective and it makes sense to consider Conjectures \ref{conj:GW/GV g=0} and \ref{conj:GW/GV g=1}. In this setting, the conjectures were verified in some cases for irreducible curve classes in \cite{CMT2}. 

One of the main goals of this paper is to provide more verifications for these local geometries for more general low degree curve classes.

\subsection{Stable pair invariants of local surfaces}

Let $S$ be a smooth projective surface and let $L_1, L_2$ be two line bundles on $S$ satisfying $L_1 \otimes L_2 \cong K_S$. Then the total space $X$ of $L_1 \oplus L_2$ over $S$ is a non-proper Calabi-Yau 4-fold, which we refer to as a \emph{local surface}. Consider the moduli space $P_n(X,\beta)$ of stable pairs $(F,s)$ with $\chi(F)=n$ and such that $F$ has proper scheme theoretic support in class $\beta \in H_2(X,\mathbb{Z})$. Although $P_n(X,\beta)$ is in general non-proper, it can be proper in several interesting cases (Propositions \ref{proper crit 1}, \ref{listproj}). Then we can define virtual classes \eqref{pair moduli vir class intro} and 
corresponding stable pair invariants \eqref{inv cpt case}.

\begin{exam}
For  $(S,L_1, L_2) = (\mathbb{P}^2,\oO(-1), \oO(-2))$ and $(\mathbb{P}^1 \times \mathbb{P}^1, \oO(-1,-1), \oO(-1,-1))$, 
the moduli space $P_n(X,\beta)$ is projective \emph{for all} $n,\beta$ (see Proposition \ref{proper crit 1}). 
\end{exam}

\begin{exam} \label{P1xP1 cpt/noncpt}
For $(S,L_1, L_2) = (\mathbb{P}^1 \times \mathbb{P}^1, \oO(-1,0), \oO(-1,-2))$, $P_n(X,\beta)$ is in general non-proper. E.g.~let $H_1:= \{\mathrm{pt}\} \times \mathbb{P}^1$, take $\beta = [H_1]$, and $n = \chi(\oO_{H_1})=1$. Then $N_{H_1/X} \cong \oO \oplus \oO \oplus \oO(-2)$ has sections in the first fibre direction, so $H_1 \subseteq \mathbb{P}^1 \times \mathbb{P}^1 \subseteq X$ can move off the zero section $\mathbb{P}^1 \times \mathbb{P}^1 \subseteq X$ and $P_1(X,[H_1])$ is non-proper. On the other hand, for $H_2:=  \mathbb{P}^1 \times \{\mathrm{pt}\}$ and $\beta = [H_2]$, we have $\beta \cdot L_1 < 0$ and $\beta \cdot L_2 < 0$, so $P_1(X,[H_2])$ is projective by Proposition \ref{proper crit 1}.
\end{exam}

When $S$ is toric, the local surface $X$ is toric and the vertex formalism for calculating stable pair invariants of $X$ has been developed in \cite{CK2, CKM} in analogy with \cite{PT2}. Let $T \subseteq (\mathbb{C}^*)^4$ denote the 3-dimensional subtorus preserving the Calabi-Yau volume form, then the fixed locus $P_n(X,\beta)^T$ consists of finitely many isolated reduced points \cite[Sect.~2.2]{CK2}, though the number of fixed points is typically very large making calculations using the vertex formalism cumbersome. 

Although we perform a few new calculations using the vertex formalism as well, 
we mainly focus on another approach, where we use the global geometry of $S$. 
We consider the case when all stable pairs on $X$ are scheme theoretically supported on the zero section $\iota:S\hookrightarrow X$, i.e. we have an isomorphism
$$
\iota_*:P_n(S,\beta) \cong P_n(X,\beta). 
$$
Under this isomorphism, we have (Proposition \ref{compare virtual class})
\begin{equation}\label{intro choice of ori}
[P_n(X,\beta)]^{\mathrm{vir}}=(-1)^{\beta \cdot L_2 +n}\cdot e\big(-\dR\hH om_{\pi_{P_{S}}}(\mathbb{F}, \mathbb{F} \boxtimes L_1)\big) \cdot [P_n(S,\beta)]^{\mathrm{vir}},  
\end{equation}
where $[P_n(S,\beta)]^{\mathrm{vir}}$ is the virtual class of the pairs obstruction theory on $S$, $e(\cdot)$ denotes Euler class, $\pi_{P_{S}} \colon S \times P_n(S, \beta) \to P_n(S, \beta)$ is the projection, $\dR\hH om_{\pi_{P_S}} = \dR \pi_{P_S*} \circ \dR\hH om$, and $\mathbb{F}$ is the universal 1-dimensional sheaf on $S \times P_n(S, \beta)$.
The sign $(-1)^{\beta \cdot L_2 +n}=(-1)^{\beta\cdot c_1(Y) +n}$, where $Y=\mathrm{Tot}_S(L_1)$, comes from a preferred choice of orientation on $P_n(X,\beta)$ which was discussed in a similar situation in \cite{CaoFano}.

In order to use \eqref{intro choice of ori} for calculations, we need the fact that $P_n(S,\beta)$ is isomorphic to a relative Hilbert scheme. 
More precisely, assume $b_1(S) = 0$ and 
denote by $|\beta|$ the linear system determined by $\beta$. 
Denote by $\cC \rightarrow |\beta|$ the universal curve, then \cite[Prop.~B.8]{PT3} gives
$$
P_n(S,\beta) \cong \Hilb^m(\cC / |\beta|),
$$
where $\Hilb^m(\cC / |\beta|)$ denotes the relative Hilbert scheme of $m$ points on the fibres of $\cC \rightarrow |\beta|$ and
$$
m = n + g(\beta) - 1 = n + \tfrac{1}{2}\beta(\beta + K_S).
$$
This isomorphism was exploited in order to determine the surface contribution to stable pair invariants of local surfaces $\mathrm{Tot}_S(K_S)$ in \cite{KT2}. The relative Hilbert scheme $\Hilb^m(\cC / |\beta|)$ is an incidence locus in a smooth ambient space
$$
\Hilb^m(\cC / |\beta|) \subseteq S^{[m]} \times |\beta|,
$$
where $S^{[m]}$ denotes the Hilbert scheme of $m$ points on $S$. More precisely, $\Hilb^m(\cC / |\beta|)$ is cut out tautologically by a section of a vector bundle on $S^{[m]} \times |\beta|$ as we recall in Section \ref{sec:relHilb}. This allows us to express the stable pair invariants of $X$ in terms of intersection numbers on $S^{[m]} \times |\beta|$, or more precisely, on the ``virtual'' ambient space $S^{[m]} \times \mathbb{P}^{\chi(\beta)-1}$, where 
$$\chi(\beta) := \chi(\oO_S(\beta)). $$
In what follows, $\zZ \subseteq S \times S^{[m]}$ denotes the universal subscheme and $\iI$ is the corresponding ideal sheaf. For any line bundle $\lL$ on $S$, the corresponding tautological bundle is defined by
$$
\lL^{[m]} := p_* q^* \lL,
$$
where $p : \zZ \rightarrow S^{[m]}$ and $q : \zZ \rightarrow S$ are projections. 
Moreover, we consider the ``twisted tangent bundle'' \cite{CO}
\begin{equation} \label{twisted tangent bundle}
T_{S^{[m]}}(\lL) := \dR \Gamma(S,\lL) \otimes \oO -\dR\hH om_{\pi}(\iI, \iI \boxtimes \lL),
\end{equation}
where $\pi : S \times S^{[m]} \rightarrow S^{[m]}$ denotes projection. 
Finally, we denote the total Chern class by $c$ and the tautological line bundle on $\mathbb{P}^{\chi(\beta)-1}$ by $\oO(1)$. We prove the following result (Theorem \ref{mainthm}).
\begin{thm} \label{mainthm intro} 
Let $S$ be a smooth projective surface with $b_1(S) = p_g(S) = 0$ and $L_1, L_2 \in \Pic(S)$ such that $L_1 \otimes L_2 \cong K_S$. 
Suppose $\beta \in H_2(S,\mathbb{Z})$ and $n \geqslant 0$ are chosen such that $P_n(X,\beta) \cong P_n(S,\beta)$ for $X = \mathrm{Tot}_S(L_1 \oplus L_2)$. 
Denote by $[\mathrm{pt}] \in H^4(X,\mathbb{Z})$ the pull-back of the Poincar\'e dual of the point class on $S$.
Let $P_n(X,\beta)$ be endowed with the orientation as in \eqref{intro choice of ori}. Then
\begin{align*}
P_{n,\beta}([\mathrm{pt}])&=(-1)^{\beta \cdot L_2+n} \int_{S^{[m]}\times \mathbb{P}^{\chi(\beta)-1}} c_{m}(\oO_S(\beta)^{[m]}(1)) \, \frac{h^n(1+h)^{\chi(L_1(\beta))} (1-h)^{\chi(L_2(\beta))} \, c(T_{S^{[m]}}(L_1))}{c(L_1(\beta)^{[m]} (1))\cdot c((L_2(\beta)^{[m]}(1))^{\vee})},
\end{align*}
when $\beta^2 \geqslant 0$. Here $m := n + g(\beta) - 1$ and $h := c_1(\oO(1))$. Moreover, $P_{n,\beta}([\mathrm{pt}]) = 0$ when $\beta^2<0$.
\end{thm}

The main assumption in this theorem is $P_n(X,\beta) \cong P_n(S,\beta)$. For $(S,L_1,L_2)$ with $S$ minimal and toric, $L_1 \otimes L_2 \cong K_S$ with $L_1^{-1}, L_2^{-1}$ non-trivial and nef, we classify all cases for which $n \geqslant 0$,  $P_n(X,\beta) \cong P_n(S,\beta)$, and $P_n(S,\beta)$ is non-empty (Proposition \ref{stable pair on del pezzo}, Remark \ref{CKK for nef cases}).
Note that $P_n(X,\beta) \cong P_n(S,\beta)$ more or less forces $p_g(S) = 0$, because as soon as $L_1$ or $L_2$ has non-zero sections this isomorphism does not hold. See Remark \ref{rmk on b1>0} for 
an extension to the case $b_1(S)>0$.

\subsection{Verifications}

In this paper, we apply Theorem \ref{mainthm intro} to examples for which $S$ is in addition toric\,\footnote{To our knowledge, all local surfaces which are Calabi-Yau 4-folds and for which Gopakumar-Vafa type invariants have been calculated so far are toric.}.
Then the integrals on $S^{[m]}$ of Theorem \ref{mainthm intro} can be calculated using Atiyah-Bott localization for the lift of the 2-dimensional torus action from $S$ to $S^{[m]}$ as described in Section \ref{sec:torusloc}. This leads to the tables for stable pair invariants in Appendix \ref{tables}.

Denote by $[H] \in H_2(\mathbb{P}^2, \mathbb{Z})$ the class of a line and let $[H_1], [H_2] \in H_2(\mathbb{P}^1 \times \mathbb{P}^1,\mathbb{Z})$ be as in Example \ref{P1xP1 cpt/noncpt}. In \cite[Sect.~3]{KP}, Klemm-Pandharipande determined the Gromov-Witten invariants of $X = \mathrm{Tot}_S(L_1 \oplus L_2)$ 
for $(S,L_1, L_2) = (\mathbb{P}^2,\oO(-1), \oO(-2))$ and $(\mathbb{P}^1 \times \mathbb{P}^1,\oO(-1,-1), \oO(-1,-1))$. They tabulated the corresponding values of the Gopakumar-Vafa type invariants for $\beta = d[H]$ with $d \leqslant 10$ resp.~$\beta = d_1[H_1]+d_2[H_2]$ with $d_1, d_2 \leqslant 6$. 
Combining their calculations and the tables in Appendix \ref{tables}, we deduce the following:
\begin{cor} \label{main cor 1}
In the following cases, Conjectures \ref{conj:GW/GV g=0} and \ref{conj:GW/GV g=1} are true for $X = \mathrm{Tot}_S(L_1 \oplus L_2)$.
\begin{itemize}
\item $(S,L_1,L_2) = (\mathbb{P}^2,\oO(-1),\oO(-2))$, $d=1$, and any $n \geqslant 0$.
\item $(S,L_1,L_2) = (\mathbb{P}^2,\oO(-1),\oO(-2))$, $d=2,3,4$, and $n=0,1$.
\item $(S,L_1,L_2) = (\mathbb{P}^2,\oO(-1),\oO(-2))$, $d=2,3$, and $n=2$.
\item $(S,L_1,L_2) = (\mathbb{P}^1 \times \mathbb{P}^1,\oO(-1,-1), \oO(-1,-1))$, $(d_1,d_2)=(1,0),(0,1),(1,1)$, any $n \geqslant 0$.
\item $(S,L_1,L_2) = (\mathbb{P}^1 \times \mathbb{P}^1,\oO(-1,-1), \oO(-1,-1))$, $(d_1,d_2)=(0,d), (d,0)$ with $d \geqslant 2$, and $0\leqslant n\leqslant d$.
\item $(S,L_1,L_2) = (\mathbb{P}^1 \times \mathbb{P}^1,\oO(-1,-1), \oO(-1,-1))$, $(d_1,d_2)=(1,d), (d,1)$ with $d \geqslant 2$, and $n=0,1,2$.
\item $(S,L_1,L_2) = (\mathbb{P}^1 \times \mathbb{P}^1,\oO(-1,-1), \oO(-1,-1))$, $(d_1,d_2)=(2,2)$, $(2,3)$, $(3,2)$, $(2,4)$, $(4,2)$, $(3,3)$, and $n=0$.
\item $(S,L_1,L_2) = (\mathbb{P}^1 \times \mathbb{P}^1,\oO(-1,-1), \oO(-1,-1))$, $(d_1,d_2)=(2,2), (2,3), (3,2)$, and $n=1$.
\item $(S,L_1,L_2) = (\mathbb{P}^1 \times \mathbb{P}^1,\oO(-1,-1), \oO(-1,-1))$, $(d_1,d_2)=(2,2)$, and $n=2$.
\end{itemize}
\end{cor}

\begin{rmk}
In all these cases $P_n(X,\beta) \cong P_n(S,\beta)$. In fact, these are \emph{all} $(S,L_1,L_2)$ with $L_1 \otimes L_2 \cong K_S$ for which $L_1^{-1}, L_2^{-1}$ are ample, $n \geqslant 0$, and $P_n(X,\beta) \cong P_n(S,\beta)$ by Propositions  \ref{CMT lemma} and \ref{stable pair on del pezzo}. Calculations based on Theorem \ref{mainthm intro} are often more efficient than the vertex formalism \cite{CK2, CKM}. For instance, for $(S,L_1,L_2) = (\mathbb{P}^1 \times \mathbb{P}^1, \oO(-1,-1), \oO(-1,-1))$, $(d_1,d_2)=(2,4)$ and $n=0$, $P_n(X,\beta)$ has $182$ $T$-fixed points, whereas Theorem \ref{mainthm intro} only involves an integral over $S^{[2]} \times \mathbb{P}^{14}$. 
\end{rmk}

Bousseau-Brini-van Garrel \cite{BBG} recently determined the genus zero Gromov-Witten (and hence Gopakumar-Vafa type) invariants of several local surfaces for their verifications of the log-local principle conjectured in general in \cite{GGR}. Combining their numbers with the tables for stable pair invariants in Appendix \ref{tables} allows us to provide some further verifications of Conjecture \ref{conj:GW/GV g=0} as we will now describe.
For any $a \geqslant 1$, consider the Hirzebruch surface
$$
\mathbb{F}_a = \mathbb{P}(\oO_{\mathbb{P}^1} \oplus \oO_{\mathbb{P}^1}(a)).
$$
We denote by $[F]$ the class of a fibre and by $[B]$ the class of the unique section satisfying $B^2 = -a$. We write $\oO(m,n):=\oO(m B + n F)$ and consider curve classes $\beta := d_1 [B] + d_2 [F]$, $d_1,d_2 \geqslant 0$. 
\begin{cor} \label{main cor 2}
In the following cases, Conjecture \ref{conj:GW/GV g=0} is true for $X = \mathrm{Tot}_S(L_1 \oplus L_2)$. 
\begin{itemize}
\item $(S,L_1,L_2) = (\mathbb{P}^1 \times \mathbb{P}^1,\oO(-1,0), \oO(-1,-2))$, $(d_1,d_2)=(0,1)$, and any $n \geqslant 1$.
\item $(S,L_1,L_2) = (\mathbb{P}^1 \times \mathbb{P}^1,\oO(-1,0), \oO(-1,-2))$, $(d_1,d_2)=(0,d)$ with $d \geqslant 2$, and $n=d$.
\item $(S,L_1,L_2) = (\mathbb{P}^1 \times \mathbb{P}^1,\oO(-1,0), \oO(-1,-2))$, $(d_1,d_2)=(2,2),(2,3),(1,d),(d,1)$ with $d \geqslant 1$, and $n=1$.
\item $(S,L_1,L_2) = (\mathbb{P}^1 \times \mathbb{P}^1,\oO(-1,0), \oO(-1,-2))$, $(d_1,d_2)=(1,d)$ with $d \geqslant 2$, and $n=2$.
\item $(S,L_1,L_2) = (\mathbb{F}_1,\oO(-1,-1), \oO(-1,-2))$, $(d_1,d_2)=(0,1)$, and any $n \geqslant 1$.
\item $(S,L_1,L_2) = (\mathbb{F}_1,\oO(-1,-1), \oO(-1,-2))$, $(d_1,d_2)=(0,d)$ with $d \geqslant 2$, and $n=d$.
\item $(S,L_1,L_2) = (\mathbb{F}_1,\oO(-1,-1), \oO(-1,-2))$, $(d_1,d_2)=(2,2),(2,3),(2,4),(1,d)$ with $d \geqslant 1$, and $n=1$.
\item $(S,L_1,L_2) = (\mathbb{F}_1,\oO(-1,-1), \oO(-1,-2))$, $(d_1,d_2)=(1,d)$ with $d \geqslant 2$, and $n=2$.
\item $(S,L_1,L_2) = (\mathbb{F}_1,\oO(0,-1), \oO(-2,-2))$, $(d_1,d_2)=(2,2),(2,3),(1,d)$ with $d \geqslant 1$, $n=1$.
\item $(S,L_1,L_2) = (\mathbb{F}_2,\oO(-1,-2), \oO(-1,-2))$, $(d_1,d_2)=(0,1)$, and any $n \geqslant 1$.
\item $(S,L_1,L_2) = (\mathbb{F}_2,\oO(-1,-2), \oO(-1,-2))$, $(d_1,d_2)=(0,d)$ with $d \geqslant 2$, and $n=d$.
\item $(S,L_1,L_2) = (\mathbb{F}_2,\oO(-1,-2), \oO(-1,-2))$, $(d_1,d_2)=(2,3),(2,4),(2,5),(1,d)$ with $d \geqslant 1$, and $n=1$.
\item $(S,L_1,L_2) = (\mathbb{F}_2,\oO(-1,-2), \oO(-1,-2))$, $(d_1,d_2)=(1,d)$ with $d \geqslant 2$, and $n=2$.
\end{itemize}
\end{cor}
In all these cases $P_n(X,\beta) \cong P_n(S,\beta)$ and $n>0$. Since Bousseau-Brini-van Garrel only determined the \emph{genus zero} Gopakumar-Vafa type invariants for the above geometries, we can only verify Conjecture \ref{conj:GW/GV g=0} in these cases. In fact, in Proposition \ref{stable pair on del pezzo} and Remark \ref{CKK for nef cases}, we classify \emph{all} cases $(S,L_1,L_2)$ such that $S$ is minimal toric, $L_1 \otimes L_2 \cong K_S$, $L_1^{-1}, L_2^{-1}$ are non-trivial and nef, $n \geqslant 0$, $P_n(X,\beta) \cong P_n(S,\beta)$, and $P_n(S,\beta)$ is non-empty. Using Theorem \ref{mainthm intro}, we determined the stable pair invariants in all these cases, \emph{including} the $n=0$ case (see Appendix \ref{tables}).

\begin{rmk}
For all calculations done in Appendix \ref{tables} for which $P_n(X,\beta) \cong P_n(S,\beta)$ and the invariant is non-zero, we have
\begin{equation} \label{COnumbers}
P_{n,\beta}([\pt]) = \pm \int_{S^{[m]}} e(T_{S^{[m]}}(L_1)).
\end{equation}
These numbers were calculated by Carlsson-Okounkov \cite{CO} and are determined by the formula
$$
\sum_{m=0}^{\infty} q^m \int_{S^{[m]}} e(T_{S^{[m]}}(L_1)) = \prod_{m=1}^{\infty} (1-q^m)^{-c_2(T_S \otimes L_1)},
$$
where $c_2(T_S \otimes L_1) = c_2(S) - L_1L_2$. We do not know whether \eqref{COnumbers} is a mere coincidence. 
\end{rmk}

\subsection{Vertex calculations}

Although most calculations in this paper are based on Theorem \ref{mainthm intro}, we also did some computations using the vertex formalism. 
\begin{prop} \label{prop intro}
For the following cases, Conjectures \ref{conj:GW/GV g=0} and \ref{conj:GW/GV g=1} are true.
 \begin{itemize}
 \item $X =\mathrm{Tot}_{\mathbb{P}^3}(K_{\mathbb{P}^3})$, $d=1$, and any $n \geqslant 0$.
\item $X =\mathrm{Tot}_{\mathbb{P}^3}(K_{\mathbb{P}^3})$, $d=2,3$, and $n=0,1$.
\item $X =\mathrm{Tot}_{\mathbb{P}^3}(K_{\mathbb{P}^3})$, $d=2$, and $n=2$.
\end{itemize}
For the following cases Conjecture \ref{conj:GW/GV g=0} is true for $X = \mathrm{Tot}_S(L_1 \oplus L_2)$. 
\begin{itemize}
 \item $(S,L_1,L_2) = (\mathbb{P}^2,\oO(-1),\oO(-2))$, $d=2$, and $n=3$. 
 \item $(S,L_1,L_2) = (\mathbb{P}^1 \times \mathbb{P}^1,\oO(-1,-1), \oO(-1,-1))$, $(d_1,d_2)=(0,2), (2,0), (1,2), (2,1)$, and $n=3$.
 \item $(S,L_1,L_2) = (\mathbb{F}_1,\oO(-1,-1), \oO(-1,-2))$, $(d_1,d_2)=(0,2)$, and $n=3$.
 \end{itemize}
\end{prop}
The invariants in this proposition are defined by localization on the fixed locus \cite{CK2, CKM}. In the cases above where $X =\mathrm{Tot}_{\mathbb{P}^3}(K_{\mathbb{P}^3})$, we have $P_n(X,\beta) \cong P_n(\mathbb{P}^3,\beta)$ and 
\begin{equation*}
[P_n(X,\beta)]^{\mathrm{vir}}=(-1)^{\beta \cdot c_1(\mathbb{P}^3) +n}\cdot [P_n(\mathbb{P}^3,\beta)]_{\mathrm{pair}}^{\mathrm{vir}},  
\end{equation*}
where $[P_n(\mathbb{P}^3,\beta)]_{\mathrm{pair}}^{\mathrm{vir}}$ is the virtual class of the pairs perfect obstruction theory on $\mathbb{P}^3$ (discussed in \eqref{pairpot}, see also \cite[Lem.~3.1]{CMT2} in a similar setting). The sign in this formula is a preferred choice of orientation on $P_n(X,\beta)$ similar to \eqref{intro choice of ori}. Then the Graber-Pandharipande virtual localization formula \cite{GP} can be applied to the right hand side to show that the local invariants of Proposition \ref{prop intro} are equal to the global invariants \eqref{inv cpt case}. The same method works for the local surface case $(S,L_1,L_2) = (\mathbb{P}^2,\oO(-1),\oO(-2))$, $d=2$, $n=3$, because then all stable pairs are scheme theoretically supported in the threefold $\mathrm{Tot}_S(L_1)$.\,\footnote{For the other local surfaces cases of Proposition \ref{prop intro}, equating our invariants to the global invariants requires a more general virtual localization formula. Recently, Oh-Thomas announced such a formula  \cite{OT}.} See Remark \ref{virtloc} for more details.

We remark that most stable pair invariants of local surfaces calculated in this paper are small (see Section \ref{local surfa list}).
For $X =\mathrm{Tot}_{\mathbb{P}^3}(K_{\mathbb{P}^3})$, the numbers are rather big:
\begin{align*}
P_{0,3[\ell]} = 11200, \quad P_{1,2[\ell]}([\ell])= -820, \quad P_{1,3[\ell]}([\ell]) = -68060, \quad P_{2,2[\ell]}([\ell])= 400,
\end{align*}
where $[\ell] \in H_2(\mathbb{P}^3,\mathbb{Z}) \cong H_2(X,\mathbb{Z})$ denotes the class of a line $\ell \subseteq \mathbb{P}^3$ and we also write
$[\ell] \in H^4(X,\mathbb{Z})$ for the pull-back of its Poincar\'e dual from $\mathbb{P}^3$ to $X$.
This provides further good evidence for Conjectures \ref{conj:GW/GV g=0} and \ref{conj:GW/GV g=1}.

\subsection{Acknowledgements}
We warmly thank Pierrick Bousseau, Andrea Brini, and Michel van Garrel for providing their values of the genus zero Gopakumar-Vafa type invariants for the geometries listed in Corollary \ref{main cor 2} and for useful discussions. This allowed us to use our calculations of stable pair invariants for verifying Conjecture \ref{conj:GW/GV g=0} in more cases.
We are grateful to Davesh Maulik and Yukinobu Toda for various helpful discussions and communications.
We also thank the anonymous referee for several helpful comments, which improved the exposition of the paper.
Y.C. is partially supported by the World Premier International Research Center Initiative (WPI), MEXT, Japan, the JSPS KAKENHI Grant Number JP19K23397,
and Royal Society Newton International Fellowships Alumni 2019. 
M.K. is supported by NWO grant VI.Vidi.192.012.
S.M. is supported by NWO grant TOP2.17.004.

\section{Background}

\subsection{DT invariants of Calabi-Yau 4-folds}\label{review dt4}

Let $X$ be a smooth projective Calabi-Yau 4-fold with ample divisor $\omega$
and take a cohomology class
$v \in H^{\ast}(X, \mathbb{Q})$.
The coarse moduli space $M_{\omega}(v)$
of $\omega$-Gieseker semistable sheaves
$E$ on $X$ with $\ch(E)=v$ is a projective scheme.
We always assume that
$M_{\omega}(v)$ is a fine moduli space, i.e.~any point $[E] \in M_{\omega}(v)$ is stable and
there exists a universal family $\eE$ on $X \times M_{\omega}(v)$ flat over $M_\omega(v)$.
For instance, the moduli space of 1-dimensional stable sheaves $E$ with $[E]=\beta$, $\chi(E)=1$ and Hilbert schemes of 
closed subschemes satisfy this assumption \cite{CK1, CK2, CMT1, CT2}.

Borisov-Joyce~\cite{BJ} (in general) and Cao-Leung \cite{CL} (in special cases) constructed a virtual fundamental class
\begin{align}\label{virtual}
[M_{\omega}(v)]^{\rm{vir}} \in H_{2-\chi(v, v)}(M_{\omega}(v), \mathbb{Z}), \end{align}
where $\chi(\cdot,\cdot)$ denotes the Euler pairing.
%Notice that this class could (a priori) be non-algebraic.
In order to construct the above virtual class (\ref{virtual}) with coefficients in $\mathbb{Z}$ (instead of $\mathbb{Z}_2$), we need an orientability result 
for $M_{\omega}(v)$, which can be stated as follows. Let  
\begin{equation*}
 \lL:=\mathrm{det}(\dR \hH om_{\pi_M}(\eE, \eE))
 \in \Pic(M_{\omega}(v)), \quad  
\pi_M \colon X \times M_{\omega}(v)\to M_{\omega}(v)
\end{equation*}
be the determinant line bundle of $M_{\omega}(v)$, which is equipped with the nondegenerate symmetric pairing $Q$ induced by Serre duality.  An \textit{orientation} of 
$(\mathcal{L},Q)$ is a reduction of its structure group from $\mathrm{O}(1,\mathbb{C})$ to $\mathrm{SO}(1, \mathbb{C})=\{1\}$. In other words, we require a choice of square root of the isomorphism
\begin{align}\label{orie}Q: \lL\otimes \lL \to \oO_{M_{\omega}(v)}.  \end{align}
Existence of orientations was first proved when the Calabi-Yau 4-fold $X$ satisfies 
$\mathrm{Hol}(X)=\mathrm{SU}(4)$ and $H^{\rm{odd}}(X,\mathbb{Z})=0$ in \cite{CL2}, 
and was recently generalized to arbitrary Calabi-Yau 4-folds in \cite[Cor.~1.17]{CGJ}.
Notice that the collection of orientations forms a torsor for $H^{0}(M_{\omega}(v),\mathbb{Z}_2)$. The virtual class \eqref{virtual} depends on the choice of orientation, but we suppress it from the notation.

Roughly speaking, in order to construct \eqref{virtual}, one chooses at
every point $[E]\in M_{\omega}(v)$, a half-dimensional real subspace
\begin{align*}\Ext_{+}^2(E, E)\subseteq \Ext^2(E, E)\end{align*}
of the usual obstruction space $\Ext^2(E, E)$, on which the quadratic form $Q$ defined by Serre duality is real and positive definite. 
Then one glues local Kuranishi-type models of the form 
\begin{equation}\kappa_{+}=\pi_+\circ\kappa: \Ext^{1}(E,E)\to \Ext_{+}^{2}(E,E),  \nonumber \end{equation}
where $\kappa$ is the Kuranishi map for $M_{\omega}(v)$ at $[E]$ and $\pi_+$ denotes projection 
on the first factor of the decomposition $\Ext^{2}(E,E)=\Ext_{+}^{2}(E,E)\oplus\sqrt{-1}\cdot\Ext_{+}^{2}(E,E)$.  

In \cite{CL}, local models are glued in three special cases: 
\begin{enumerate}
\item when $M_{\omega}(v)$ consists of locally free sheaves only, 
\item  when $M_{\omega}(v)$ is smooth,
\item when $M_{\omega}(v)$ is a shifted cotangent bundle of a quasi-smooth derived scheme. 
\end{enumerate}
In each case, the corresponding virtual classes are constructed using either gauge theory or algebro-geometric perfect obstruction theory.

The general gluing construction, due to Borisov-Joyce \cite{BJ}, is 
based on Pantev-T\"{o}en-Vaqui\'{e}-Vezzosi's theory of shifted symplectic geometry \cite{PTVV} and Joyce's theory of derived $C^{\infty}$-geometry.
The corresponding virtual class is constructed using Joyce's
D-manifold theory (a machinery similar to Fukaya-Oh-Ohta-Ono's theory of Kuranishi space structures used for defining Lagrangian Floer theory).

Examples computed in this paper only involve virtual class constructions in (2) and (3) mentioned above.
We briefly review them:
\begin{itemize}
\item When $M_{\omega}(v)$ is smooth, the obstruction sheaf $\mathrm{Ob}\to M_{\omega}(v)$ is a vector bundle endowed with a quadratic form $Q$ via Serre duality. Then the virtual class is given by
\begin{equation}[M_{\omega}(v)]^{\rm{vir}}=\mathrm{PD}(e(\mathrm{Ob},Q)),   \nonumber \end{equation}
where $\mathrm{PD}(\cdot)$ denotes 
Poincar\'e dual and $e(\mathrm{Ob}, Q)$ is the half-Euler class of 
$(\mathrm{Ob},Q)$, i.e.~the Euler class of its real form $\mathrm{Ob}_+$. In this case, a choice of orientation \eqref{orie} is equivalent to a choice of
orientation of $\mathrm{Ob}_+$. 
The half-Euler class satisfies 
\begin{align*}
e(\mathrm{Ob},Q)^{2}&=(-1)^{\frac{\mathrm{rk}(\mathrm{Ob})}{2}}e(\mathrm{Ob}),  \qquad \textrm{ }\mathrm{if}\textrm{ } \mathrm{rk}(\mathrm{Ob})\textrm{ } \mathrm{is}\textrm{ } \mathrm{even}, \\
 e(\mathrm{Ob},Q)&=0, \qquad\qquad\qquad\qquad \ \,  \textrm{ }\mathrm{if}\textrm{ } \mathrm{rk}(\mathrm{Ob})\textrm{ } \mathrm{is}\textrm{ } \mathrm{odd}. 
\end{align*}
\item Suppose $M_{\omega}(v)$ is the classical truncation of the shifted cotangent bundle of a quasi-smooth derived scheme. Roughly speaking, this means that at any closed point $[E]\in M_{\omega}(v)$, we have a Kuranishi map of the form
\begin{equation}\kappa \colon
 \Ext^{1}(E,E)\to \Ext^{2}(E,E)=V_E\oplus V_E^{*},  \nonumber \end{equation}
where $\kappa$ factors through a maximal isotropic subspace $V_E$ of $(\Ext^{2}(E,E),Q)$. Then the virtual class of $M_{\omega}(v)$ is essentially the 
virtual class of the perfect obstruction theory formed by $\{V_E\}_{[E]\in M_{\omega}(v)}$. 
When $M_{\omega}(v)$ is furthermore smooth as a scheme, 
then it is
simply the Euler class of the vector bundle 
$\{V_E\}_{[E]\in M_{\omega}(v)}$ over $M_{\omega}(v)$. 
\end{itemize}

\subsection{Stable pair invariants of Calabi-Yau 4-folds}

As in \cite{PT1}, a stable pair $(F,s)$ on a smooth projective Calabi-Yau 4-fold $X$ consists of
\begin{itemize}
\item a pure dimension 1 sheaf $F$ on $X$,
\item a section $s \in H^0(X,F)$ with 0-dimensional or trivial cokernel.
\end{itemize}

For $\beta \in H_2(X, \mathbb{Z})$ and $n\in \mathbb{Z}$, denote by $P_n(X, \beta)$
be the moduli space of stable pairs $(F, s)$ on $X$
such that $F$ has scheme theoretic support with class $\beta$ and $\chi(F)=n$.
By \cite{PT1}, it can alternatively be seen as the moduli space parametrizing 2-term complexes
\begin{align*}
I^\mdot=\{\oO_X \stackrel{s}{\to} F\} \in D^b(\Coh(X))
\end{align*}
in the bounded derived category of coherent sheaves on $X$. This viewpoint produces an obstruction theory on $P_n(X, \beta)$, which is however not perfect because $\Ext^2(I^\mdot, I^\mdot)_0$ is in general non-vanishing.
Nonetheless, using the methods of Borisov-Joyce \cite{BJ}, one can construct a virtual class (see \cite[Thm.~1.4]{CMT2})
\begin{equation*}\label{pair moduli vir class}
[P_n(X, \beta)]^{\rm{vir}}\in H_{2n}\big(P_n(X, \beta),\mathbb{Z}\big)
\end{equation*}
depending on a choice of orientation. Existence of orientations was proved in \cite[Cor.~1.17]{CGJ}.

\section{Moduli spaces}

\subsection{Compactness I}

In the previous section, we assumed $X$ is a smooth \emph{projective} Calabi-Yau 4-fold. As we will discuss in more detail in Section \ref{sec:comparison}, the previous section also applies to certain cases where $X$ is a smooth quasi-projective Calabi-Yau 4-fold and $P_n(X,\beta)$ is proper. 

Suppose $S$ is a smooth projective surface and $L_1$, $L_2 \in \Pic(S)$ satisfy
$$
L_1 \otimes L_2 \cong K_S.
$$
Then $X = \mathrm{Tot}_S(L_1 \oplus L_2)$ is a smooth quasi-projective Calabi-Yau 4-fold, which we refer to as a \emph{local surface}. 
One way to ensure the properness of $P_n(X,\beta)$ is as follows.
\begin{prop} \label{proper crit 1}
Suppose $S$ is a smooth projective surface with $L_1$, $L_2 \in \Pic(S)$ satisfying $L_1 \otimes L_2 \cong K_S$ and let $X = \mathrm{Tot}_S(L_1 \oplus L_2)$. Let $\beta \in H_2(S,\mathbb{Z})$ and suppose for any $0 \neq \beta' \leqslant \beta$,\footnote{The notation $\beta' \leqslant \beta$ means that there exist effective curve classes $\beta',\beta'' \in H_2(S,\mathbb{Z})$ such that $\beta = \beta' + \beta''$.} we have $\beta' \cdot L_i < 0$ for $i=1, 2$.
Then $P_n(X,\beta)$ is projective for any $n \in \mathbb{Z}$. 
\end{prop}
\begin{proof}
Let $[(F,s)] \in P_n(X,\beta)$. We first show that $F$ is set theoretically supported on the zero section $S \subseteq X$. Let $D$ be an irreducible component of the scheme theoretic support of $F$, then we want to show $D_{\mathrm{red}} \subseteq S$. Let $Y = \mathrm{Tot}_S(L_1)$ and consider the projection $p : X = \mathrm{Tot}_Y(L_2) \rightarrow Y$ (here and below, we suppress the pull-back of $L_2$ along the projection $Y \to S$). Since $D_{\mathrm{red}}$ is a proper irreducible reduced curve, $\oO_{D_{\mathrm{red}}}$ is stable. By the spectral construction, it corresponds to a stable Higgs pair $(p_* \oO_{D_{\mathrm{red}}},\phi)$, where 
$$
\phi : p_* \oO_{D_{\mathrm{red}}} \rightarrow p_* \oO_{D_{\mathrm{red}}} \otimes L_2.
$$
Denote the curve class of the scheme theoretic support of $p_* \oO_{D_{\mathrm{red}}}$ by $\beta' \in H_2(Y,\mathbb{Z}) \cong H_2(S,\mathbb{Z})$. Then $0 \neq \beta' \leqslant \beta$, so $\beta' \cdot L_2 < 0$. Combined with stability of the Higgs pairs $(p_* \oO_{D_{\mathrm{red}}},\phi)$ and $(p_* \oO_{D_{\mathrm{red}}} \otimes L_2,\phi \otimes \mathrm{id}_{L_2})$, this implies $\phi =0$ so $D_{\mathrm{red}} \subseteq Y =  \mathrm{Tot}_S(L_1)$ (see \cite[Prop.~7.4]{TT} for a similar argument). Reversing the roles of $L_1, L_2$, we deduce $D_{\mathrm{red}} \subseteq Y =  \mathrm{Tot}_S(L_2)$, so $D_{\mathrm{red}} \subseteq S$. 

Since each element of $P_n(X,\beta)$ is set theoretically supported on $S$, we conclude that $P_n(X,\beta)$ is projective. Indeed, there is a $d \gg 0$ such that every element of $P_n(X,\beta)$ is scheme theoretically supported in $d S$, where $d S$ denotes the $d$ times thickening of the zero section $S \subseteq X$, i.e.~the closed subscheme of $X$ defined by $I^{d} \subseteq \oO_X$, where $I \subseteq \oO_X$ denotes the ideal of the zero section. Therefore $P_n(X,\beta) \cong P_n(dS,\beta)$, 
\end{proof}
Suppose $L_1^{-1}$ and $L_2^{-1}$ are ample. Then $K_S^{-1}$ is ample, i.e.~$S$ is del Pezzo, and $P_n(X,\beta)$ is projective for all $\beta, n$ by Proposition \ref{proper crit 1}. As noted in \cite[Sect.~4.2]{CMT2}, there are only two possibilities: 
\begin{prop} \label{CMT lemma}
Let $S$ be a smooth projective surface and $L_1, L_2 \in \Pic(S)$ such that $L_1 \otimes L_2 \cong K_S$. Suppose $L_1^{-1}$ and $L_2^{-1}$ are ample. Then, up to permutating $L_1, L_2$, we only have $(S,L_1,L_2) = (\mathbb{P}^2,\oO(-1),\oO(-2))$ or $(S,L_1,L_2) = (\mathbb{P}^1 \times \mathbb{P}^1,\oO(-1,-1),\oO(-1,-1))$.
\end{prop}
\begin{proof}
Suppose $S$ contains a $(-1)$-curve $C$. Then the Nakai criterion and adjunction imply
$$
-2 \geqslant \deg(L_1|_C) + \deg(L_2|_C) = \deg(K_S|_C) = -1,
$$
so $S$ does not contain $(-1)$-curves. The classification of del Pezzo surfaces yields the result.
\end{proof}
For both geometries of this proposition, the Gromov-Witten (and hence Gopakumar-Vafa type) invariants were determined in \cite[Sect.~3]{KP}.

Let us go back to an arbitrary smooth projective surface $S$ with $L_1$, $L_2 \in \Pic(S)$ satisfying $L_1 \otimes L_2 \cong K_S$. Consider the moduli space $P_n(S,\beta)$ of stable pairs $(F,s)$ on $S$ with $\chi(F) = n$ and scheme theoretic support of $F$ in class $\beta \in H_2(S,\mathbb{Z})$. Any stable pair $I^\mdot = \{ \oO_S \rightarrow F\}$ gives rise to a stable pair
$$
\{ \oO_X \rightarrow \iota_* \oO_S \rightarrow \iota_* F\}
$$
on $X = \mathrm{Tot}_S(L_1 \oplus L_2)$, where $\iota : S \hookrightarrow X$ denotes inclusion of the zero section. This gives a closed embedding 
\begin{equation} \label{emb}
P_n(S,\beta) \hookrightarrow P_n(X,\beta).
\end{equation}
We refer to elements of $P_n(X,\beta)$ in the image as ``stable pairs which are scheme theoretically supported on $S$''.
Requiring $P_n(X,\beta)$ to be proper poses restrictions on $n,\beta$.
The following result is very useful for finding ``candidates'' for proper moduli spaces $P_n(X,\beta)$ (as we will see later in this section in Proposition \ref{listproj}). 
\begin{prop} \label{properness prop}
Let $S$ be a smooth projective surface, $L_1$, $L_2 \in \Pic(S)$ such that $L_1 \otimes L_2 \cong K_S$ and let $X = \mathrm{Tot}_S(L_1 \oplus L_2)$. Let $\beta \in H_2(S,\mathbb{Z})$ and $n \in \mathbb{Z}$ such that $P_n(X,\beta)$ is proper and $P_n(S,\beta) \neq \varnothing$. Suppose $C_1, C_2 \subseteq S$ are effective divisors satisfying
\begin{itemize}
\item $C_1 \cong \mathbb{P}^1$ and $[C_1+C_2] = \beta$,
\item $L_i \cdot C_1 = 0$ for $i=1$ or $i=2$.
\end{itemize}
Then
\begin{align*}
&- \tfrac{1}{2}\beta(\beta + K_S)  \leqslant n \leqslant - \tfrac{1}{2}C_2(C_2 + K_S).
\end{align*}
\end{prop}
\begin{proof}
Suppose $P_n(S,\beta) \neq \varnothing$, $P_n(X,\beta)$ is proper, and let $C_1, C_2 \subseteq S$ be as stated. Then for any element $[(F,s)] \in P_n(S,\beta)$ with underlying scheme theoretic support $C$, we have
$$
n = \chi(F) \geqslant \chi(\oO_C) = - \tfrac{1}{2}\beta(\beta + K_S).
$$
Suppose 
$$
n \geqslant 1  - \tfrac{1}{2}C_2(C_2 + K_S).
$$
Since $C_1 \cong \mathbb{P}^1$ and $\deg(L_i|_{C_1}) = L_i \cdot C_1 = 0$, for $i=1$ or $i=2$, the line bundle $L_i|_{C_1}$ is trivial. 
Hence we can take a nowhere vanishing section $D_1$ of the line bundle $L_i|_{C_1} \cong \mathbb{P}^1 \times \mathbb{C}$. In particular, $D_1$ and $C_2$ are disjoint. 
Therefore
\begin{align*}
\chi(\oO_{D_1 \sqcup C_2}) &=  \chi(\oO_{D_1}) + \chi(\oO_{C_2})  \\
&= 1 - \tfrac{1}{2}C_2(C_2 + K_S).
\end{align*}
Twisting $\oO_{D_1 \sqcup C_2}$ by an effective divisor of appropriate length, we obtain a stable pair $[(F,s)] \in P_n(X,\beta) \setminus P_n(S,\beta)$ with underlying scheme theoretic support $D_1 \sqcup C_2$. Since $D_1$ does not lie in the zero-section, using the $\mathbb{C}^*$-scaling action on $L_i$, we get a family of stable pairs with part of the support (i.e.~$D_1$) moving off to infinity, contradicting properness of $P_n(X,\beta)$.
\end{proof}

We want to apply this proposition to smooth projective surfaces $S$ with $L_1$, $L_2 \in \Pic(S)$ such that $L_1 \otimes L_2 \cong K_S$ and $L_1^{-1}$, $L_2^{-1}$ non-trivial and nef. These surfaces were recently studied in the context of the log-local principle by Bousseau-Brini-van Garrel \cite{BBG}. In particular, they determined the genus zero Gromov-Witten (and hence Gopakumar-Vafa type) invariants of $\mathrm{Tot}_S(L_1 \oplus L_2)$ in many new cases. 

Smooth projective surfaces $S$ with $K_S^{-1}$ nef and big are called weak del Pezzo surfaces. The weak toric del Pezzo surfaces are: $\mathbb{P}^2$, $\mathbb{P}^1 \times \mathbb{P}^1$, $\mathbb{F}_1$, $\mathbb{F}_2$, or certain repeated toric blow-ups of $\mathbb{P}^2$ in at most 6 points as specified in \cite{Sat}. In this paper, we only consider the \emph{minimal} cases, i.e.~the first four cases. Using the notation for Hirzebruch surfaces from the introduction, the only possibilities for $L_1, L_2 \in \Pic(S)$ such that $L_1 \otimes L_2 \cong K_S$ with $L_1^{-1}$, $L_2^{-1}$ non-trivial and nef are (up to permutations of $L_1,L_2$):
\begin{itemize}
\item $(S,L_1,L_2)=(\mathbb{P}^2,\oO(-1),\oO(-2))$,
\item $(S,L_1,L_2)=(\mathbb{P}^1\times \mathbb{P}^1,\oO(-1,-1),\oO(-1,-1))$ or $(\mathbb{P}^1 \times \mathbb{P}^1,\oO(-1,0),\oO(-1,-2))$,
\item $(S,L_1,L_2)=(\mathbb{F}_1,\oO(-1,-1),\oO(-1,-2))$ or $(\mathbb{F}_1,\oO(0,-1), \oO(-2,-2))$,
\item $(S,L_1,L_2)=(\mathbb{F}_2,\oO(-1,-2),\oO(-1,-2))$.
\end{itemize}
%In each case the linear systems $|L_1^{-1}|$ and $|L_2^{-1}|$ are base point free
%so their generic element is smooth, by Bertini's theorem, and connected. 
%Moreover, in each case of the list $L_1 \cdot L_2 = 2$, so the generic element is a rational curve by the adjunction formula.

\begin{exam} \label{nef case 1}
Suppose $(S,L_1,L_2)=(\mathbb{P}^1\times \mathbb{P}^1, \oO(-1,0), \oO(-1,-2))$. Let $H_1 = \{\mathrm{pt}\} \times \mathbb{P}^1$ and $H_2 =  \mathbb{P}^1 \times \{\mathrm{pt}\}$. We define $(d_1,d_2) := d_1 H_1 + d_2 H_2$. For all $d_1, d_2 \in \mathbb{Z}$, $(d_1,d_2)$ is effective if and only if $d_1, d_2 \geqslant 0$. For $(0,d)$ with $d \geqslant 1$, the moduli space $P_n(X,(0,d))$ is projective for all $n$ by Proposition \ref{proper crit 1}. Now suppose $d_1>0$, $d_2 \geqslant 0$, and $n \geqslant 0$. Let $C_1 \in |H_1|$ and $C_2 \in |(d_1-1) H_1 + d_2 H_2|$. Then $L_1 \cdot C_1 = 0$ 
and the inequalities of Proposition \ref{properness prop} reduce to
$$
d_1 + d_2 -d_1d_2 \leqslant n \leqslant d_1 + 2d_2 -d_1d_2 - 1. 
$$ 
These inequalities have the following solutions:
\begin{itemize}
\item $(d_1,d_2) = (3,2), (2,d)$ with $d \geqslant 2$ and $n=0$,
\item $(d_1,d_2) = (d,1), (1,d)$ with $d \geqslant 1$ and $n=1$, or $(d_1,d_2) = (2,d)$ with $d \geqslant 1$ and $n=1$,
\item  $(1,d)$ with $2 \leqslant n \leqslant d$.
\end{itemize}
\end{exam}

\begin{exam} \label{nef case 2}
Suppose $(S,L_1,L_2)=(\mathbb{F}_1,\oO(-1,-1),\oO(-1,-2))$ and use the notation for Hirzebruch surfaces from the introduction, so $(d_1,d_2) := d_1 B + d_2 F$ for all $d_1, d_2 \in \mathbb{Z}$. Then $(d_1,d_2)$ is effective if and only if $d_1, d_2 \geqslant 0$ (this holds for all Hirzebruch surfaces). For $(0,d)$ with $d \geqslant 1$, the moduli space $P_n(X,(0,d))$ is projective for all $n$ by Proposition \ref{proper crit 1}. Suppose $d_1 > 0$, $d_2 \geqslant 0$, and $n \geqslant 0$. Let $C_1 \in |B|$ and $C_2 \in |(d_1-1) B + d_2 F|$. Then $L_1 \cdot C_1 = 0$ and the inequalities of Proposition \ref{properness prop} reduce to
$$
\tfrac{1}{2} d_1(d_1+1) - d_2(d_1-1) \leqslant n \leqslant \tfrac{1}{2} d_1 (d_1 - 1) - d_2(d_1 - 2).
$$
These inequalities have the following solutions:
\begin{itemize}
\item $(d_1,d_2) = (3,3), (2,d)$ with $d \geqslant 3$ and $n=0$,
\item $(d_1,d_2) = (2,d)$ with $d \geqslant 2$ and $n=1$,
\item $(1,d)$ with $1 \leqslant n \leqslant d$.
\end{itemize}
\end{exam}

\begin{exam} 
Suppose $(S,L_1,L_2)=(\mathbb{F}_1,\oO(0,-1),\oO(-2,-2))$. Suppose $d_1 > 0$, $d_2 \geqslant 0$, and $n \geqslant 0$. Taking $C_1,C_2$ as in Example \ref{nef case 2} leads to the same list. Additionally, we can take $d_1 \geqslant 0$, $d_2 > 0$, $n \geqslant 0$, $C_1 \in |F|$ and $C_2 \in |d_1 B + (d_2-1) F|$. Then $L_1 \cdot C_1 = 0$ and the inequalities of Proposition \ref{properness prop} reduce to
$$
\tfrac{1}{2} d_1(d_1+1) - d_2(d_1-1) \leqslant n \leqslant \tfrac{1}{2} d_1(d_1+1) - (d_2-1)(d_1-1).
$$
The solutions to these inequalities \emph{and} the ones from Example \ref{nef case 2} are:
\begin{itemize}
\item $(d_1,d_2) = (2,3), (2,4), (3,3)$ and $n=0$,
\item $(d_1,d_2) = (2,2), (2,3), (1,d)$ with $d \geqslant 1$ and $n=1$.
\end{itemize}
\end{exam}

\begin{exam} \label{nef case 4}
Suppose $(S,L_1,L_2)=(\mathbb{F}_2,\oO(-1,-2),\oO(-1,-2))$. For $(0,d)$ with $d \geqslant 1$, the moduli space $P_n(X,(0,d))$ is projective for all $n$ by Proposition \ref{proper crit 1}. Suppose $d_1 > 0$, $d_2 \geqslant 0$, and $n \geqslant 0$. Let $C_1 \in |B|$ and $C_2 \in |(d_1-1) B + d_2 F|$. Then $L_1 \cdot C_1 = 0$ and the inequalities of Proposition \ref{properness prop} reduce to
$$
d_1^2 - d_2(d_1-1) \leqslant n \leqslant (d_1-1)^2 - d_2(d_1-2).
$$
These inequalities have the following solutions:
\begin{itemize}
\item $(d_1,d_2) = (2,d)$, $d \geqslant 4$, and $n=0$,
\item $(d_1,d_2) = (2,d)$, $d \geqslant 3$, and $n=1$,
\item $(1,d)$ with $1 \leqslant n \leqslant d$.
\end{itemize}
\end{exam}

In these examples we listed, for given $(S,L_1,L_2)$, all the cases for which $n \geqslant 0$ and \emph{potentially} $P_n(X,\beta)$ is proper and $P_{n}(S,\beta)  \neq \varnothing$ (Proposition \ref{properness prop}). For $\beta = (0,d)$ with $d \geqslant 1$, $P_{n}(S,\beta)  \neq \varnothing$ if and only if $n \geqslant d$, and $P_n(X,\beta)$ is proper by Proposition \ref{proper crit 1}. For all other cases listed, $P_n(S,\beta)$ is also non-empty since $|\beta| \neq \varnothing$ and $n \geqslant \chi(\oO_C)$ for any $C \in |\beta|$. Indeed adding sufficiently many points to $C$ one obtains a stable pair $(F,s)$ on $S$ with $\chi(F) = n$. We now prove that in each of the cases listed, $P_n(X,\beta)$ is indeed proper. 

\begin{prop} \label{listproj}
In each of the cases listed in Examples \ref{nef case 1}--\ref{nef case 4}, $P_n(X,\beta)$ is projective. 
\end{prop}
\begin{proof}
We write out the proof for Example \ref{nef case 1}. The other cases are analogous. Recall that $H_1 := \{\mathrm{pt}\} \times \mathbb{P}^1$, $H_2 := \mathbb{P}^1 \times \{\mathrm{pt}\}$, $L_1 := \oO(-H_1)$, $L_2 := \oO(-H_1-2H_2)$, and $\beta:=d_1H_1+d_2H_2$. 
Suppose $n, \beta$ are as listed in Example \ref{nef case 1}. As in Proposition \ref{proper crit 1}, it is enough to show that all elements of $P_n(X,\beta)$ are set theoretically supported on $S$. Suppose $[(F,s)] \in P_n(X,\beta)$ has scheme theoretic support $C$ and let $D$ be an irreducible component of $C$  which is \emph{not} set theoretically supported on $S$. Then we claim $D_{\mathrm{red}}$ is a proper irreducible reduced curve with class $[D_{\mathrm{red}}] \in H_2(X,\mathbb{Z}) \cong H_2(S,\mathbb{Z})$ satisfying $[D_{\mathrm{red}}] \cdot L_1 \geqslant 0$ or $[D_{\mathrm{red}}] \cdot L_2 \geqslant 0$.
Indeed suppose $[D_{\mathrm{red}}] \cdot L_1 < 0$ and $[D_{\mathrm{red}}] \cdot L_2 < 0$. Using the spectral construction as in the proof of Proposition \ref{proper crit 1}, stability of $\oO_{D_{\mathrm{red}}}$, then implies $D_{\mathrm{red}} \subseteq S$ contrary to our assumption.

The only non-zero effective curve classes $\beta'$ on $S$ such that $\beta' \cdot L_1 \geqslant 0$ or $\beta' \cdot L_2 \geqslant 0$ are $\beta' = m H_1$ for some $m>0$. 
Hence there exists a $\Sigma \in |H_1|$ such that 
$$
p|_{D_{\mathrm{red}}} : D_{\mathrm{red}} \to \Sigma \subseteq S,
$$ 
where $p : X \to S$ denotes the projection. Note that $L_1|_{\Sigma} \cong \oO$ and $L_2|_{\Sigma} \cong \oO(-2)$. Since $[D_{\mathrm{red}}] \cdot L_2 < 0$, a similar argument as above shows that $D_{\mathrm{red}} \subseteq \mathrm{Tot}_{\Sigma}(L_1) \cong \mathbb{P}^1 \times \mathbb{C}$. Therefore $D_{\mathrm{red}}$ is a non-zero section of $\mathrm{Tot}_{\Sigma}(L_1) \cong \mathbb{P}^1 \times \mathbb{C}$.

Denote the irreducible components of $C$ which are \emph{not} set theoretically supported on $S$ by $D_1, \ldots, D_{\ell}$ and let $D'$ be the union of the remaining components. Above, we showed each $D_{i,\mathrm{red}} \cong \mathbb{P}^1$ and $D_{i,\mathrm{red}}$ is a non-zero section of $\mathrm{Tot}_{\Sigma_i}(L_1) \cong \mathbb{P}^1 \times \mathbb{C}$ for some $\Sigma_i \in |H_1|$. It follows that $D_{1,\mathrm{red}}, \ldots, D_{\ell,\mathrm{red}}, D'_{\mathrm{red}}$ are mutually disjoint.  Denote the multiplicity of $D_i$ at $D_{i,\mathrm{red}}$ by $\delta_i \geqslant 1$. Consider the classes $p_*[D_i] , p_*[D'] \in H_2(S,\mathbb{Z})$, where $p : X \to S$ is the projection. Then
$$
p_*[D_i] := \delta_i H_1, \quad p_*[D']:= \beta - \delta H_1,
$$
where $\delta := \sum_{i=1}^{\ell} \delta_i$. We claim
\begin{equation} \label{ineqschi}
\chi(\oO_{D_i}) \geqslant 1, \quad \chi(\oO_{D'}) \geqslant 1-g(p_*[D']) = -\tfrac{1}{2}(\beta - \delta H_1)(\beta - \delta H_1 + K_S),
\end{equation}
for all $i=1, \ldots, \ell$, where the last equality is by the Riemann-Roch formula. When $D_{i}$ (resp.~$D'$) are reduced, these inequalities are equalities. In general, since $N_{D_{i,\mathrm{red}}/X} \cong \oO \oplus \oO \oplus \oO(-2)$ and $N_{S/X} \cong L_1 \oplus L_2$ with $L_1^{-1}, L_2^{-1}$ nef, we have inequalities as stated\,\footnote{One way to see this is by using filtrations by thickenings of $D_{i,\mathrm{red}} \subseteq X$ and $S \subseteq X$ as in the proof of Proposition \ref{stable pair on del pezzo} below.}. From \eqref{ineqschi} and the fact that $D_1, \ldots, D_\ell, D'$ are mutually disjoint, we deduce
\begin{align*}
n = \chi(F) \geqslant \chi(\oO_C) &= \sum_{i=1}^{\ell} \chi(\oO_{D_i}) + \chi(\oO_{D'}) \\
&\geqslant \delta -\tfrac{1}{2}(\beta - \delta H_1)(\beta - \delta H_1 + K_S) \\
&= \delta - \tfrac{1}{2}(d_1 - \delta)(d_2-2) - \tfrac{1}{2} d_2(d_1-\delta-2).
\end{align*}
However, for each of the cases listed in Example \ref{nef case 1}, it is easy to see that $n \leqslant \delta - 1 - \tfrac{1}{2}(d_1 - \delta)(d_2-2) - \tfrac{1}{2} d_2(d_1-\delta-2)$ for all $1 \leqslant \delta \leqslant d_1$ by explicit calculation. We have reached a contradiction.
\end{proof}

\noindent \textbf{Conclusion.} For any $(S,L_1,L_2)$ with $L_1 \otimes L_2 \cong K_S$, $L_1^{-1}, L_2^{-1}$ non-trivial and nef, $S$ minimal and toric, we classified \emph{all} $n \geqslant 0$, $\beta \in H_2(S,\mathbb{Z})$ such that $P_n(X,\beta)$ is proper and $P_n(S,\beta) \neq \varnothing$.

\subsection{Compactness II} 

In the previous section, we studied properness of $P_n(X,\beta)$ for local surfaces. In particular, for 
$(S,L_1,L_2) = (\mathbb{P}^2,\oO(-1),\oO(-2))$ or $(\mathbb{P}^1 \times \mathbb{P}^1,\oO(-1,-1),\oO(-1,-1))$, the moduli space $P_n(X,\beta)$ is always proper (Proposition \ref{proper crit 1}). 
We are now interested in the cases where $P_n(X,\beta) \cong P_n(S,\beta)$, i.e.~the embedding \eqref{emb} is an isomorphism. For the 3-fold $\mathrm{Tot}(K_{\mathbb{P}^2})$, this question was considered by Choi-Katz-Klemm in \cite[Prop.~2]{CKK}. In the proof of the following proposition, we use some of their techniques (adapted to the 4-fold setting).
\begin{prop}\label{stable pair on del pezzo}
Let $X=\mathrm{Tot}_{\mathbb{P}^2}(\oO(-1)\oplus \oO(-2))$, $\beta=d [H]$ with $d \geqslant 1$, and $n\geqslant 0$. Then
\begin{align*}
P_n(X, \beta) \cong P_n(\mathbb{P}^2, \beta) 
\end{align*}
if and only if 
\begin{enumerate}
\item $d=1$ and any $n \geqslant 0$, or
\item  $d=2,3,4$ and $n=0,1$, or
\item $d=2,3$ and $n=2$.
\end{enumerate}
Let $X=\mathrm{Tot}_{\mathbb{P}^1\times \mathbb{P}^1}(\oO(-1,-1)\oplus \oO(-1,-1))$, $\beta=d_1[H_1]+d_2[H_2] \neq 0$ with $d_1,d_2 \geqslant 0$, and $n\geqslant 0$.  Then 
\begin{align*} 
P_n(X, \beta) \cong P_n(\mathbb{P}^1\times \mathbb{P}^1, \beta), 
\end{align*}
if and only if
\begin{enumerate}
\item $(d_1,d_2) = (1,0), (0,1), (1,1)$ and any $n \geqslant 0$, or
\item $(d_1,d_2) = (0,d), (d,0)$ with $d \geqslant 2$ and $0\leqslant n\leqslant d$, or
\item $(d_1,d_2) = (1,d), (d,1)$ with $d \geqslant 2$ and $n=0,1,2$, or
\item $(d_1,d_2) = (2,2), (2,3), (3,2), (2,4), (4,2), (3,3)$ and $n=0$, or
\item $(d_1,d_2) = (2,2), (2,3), (3,2)$ and $n=1$, or
\item $(d_1,d_2) = (2,2)$ and $n=2$.
\end{enumerate}
\end{prop}
\begin{proof}
Let $[(F,s)] \in P_n(X, \beta)$ be a stable pair with scheme theoretic support $C:=\mathrm{supp}(F)$. The stable pair $(F,s)$ is set theoretically supported on the zero section $S \subseteq X$ by Proposition \ref{proper crit 1}. Let $Y_i=\mathrm{Tot}_{S}(L_i)$ for $i=1,2$. We consider the ideals of $C \subseteq X$ and $Y_1 \subseteq X$:
\begin{align*}
J:= I_{C \subseteq X} \subseteq \oO_X, \quad I_2:= I_{Y_1 \subseteq X} \subseteq \oO_X.
\end{align*} 
Note that $I_2$ is a line bundle on $X$. Since $(F,s)$ is set theoretically supported on $S \subseteq X$ (and therefore $Y_1 \subseteq X$), there exists an $\ell \geqslant 0$ such that $J+I_2^{\ell+1}=J$ and we have
\begin{align}\label{equ sum of chi}
\chi(\oO_C)=\sum_{j=0}^{\ell} \chi\Big(\frac{J+I_2^j}{J+I_2^{j+1}}\Big).  
\end{align}
For each $j$, we have a surjective map 
\begin{align*} 
p^* L_2^{-j}\cong\frac{I_2^j}{I_2^{j+1}}\to \frac{J+I_2^j}{J+I_2^{j+1}},  
\end{align*}
where $p : Y_1 \rightarrow S$ denotes projection. Hence $\frac{J+I_2^j}{J+I_2^{j+1}} \cong  \oO_{C_j} \otimes p^* L_2^{-j}$ for some closed subscheme $C_j \subseteq Y_1$ of dimension $\leqslant 1$. 
Moreover, we have $C_j \supseteq C_{j+1}$ for all $j$.\footnote{This follows from the natural surjection
$\oO_{C_{j}} \otimes p^* L_2^{-j-1} \cong \frac{J + I_2^j}{J + I_2^{j+1}} \otimes \frac{I_2}{I_2^2} \twoheadrightarrow \frac{J+I_2^{j+1}}{J+I_2^{j+2}} \cong \oO_{C_{j+1}} \otimes p^* L_2^{-j-1}$.}
From the fact that $C$ is Cohen-Macaulay, it also follows that, when non-empty, $C_j$ is not 0-dimensional. 

For a fixed $j$, we consider the ideals of $C_j \subseteq Y_1$ and $S \subseteq Y_1$:
\begin{align*}
J_j:= I_{C_j \subseteq Y_1} \subseteq \oO_{Y_1}, \quad I_1:= I_{S \subseteq Y_1} \subseteq \oO_{Y_1}.
\end{align*} 
Note that $I_1$ is a line bundle on $Y_1$. As above, there exists an $\ell_{j} \geqslant 0$ such that $J_j+I_1^{\ell_j+1}=J_j$ and we have 
\begin{align*}
\chi(\oO_{C_{j}})=\sum_{i=0}^{\ell_j} \chi\Big(\frac{J_j+I_1^i}{J_j+I_1^{i+1}}\Big). 
\end{align*}
As above, for all $i$, we have $\frac{J_j+I_1^i}{J_j+I_1^{i+1}} \cong \oO_{C_{ij}} \otimes L_1^{-i}$ for some closed subscheme $C_{ij} \subseteq S$ of dimension $\leqslant 1$. As above, we also have $C_{ij} \supseteq C_{i+1,j}$ for all $i$. This time, we leave open the possibility that $C_{ij}$ is 0-dimensional, because $C_j$ need not be Cohen-Macaulay. Nonetheless, denoting $\beta_{ij} := [C_{ij}]$, we have
$$
\beta = \sum_{j=0}^{\ell} \sum_{i=0}^{\ell_j} \beta_{ij} \in H_2(S,\mathbb{Z}).
$$
Consider the torsion filtration
$$
0\to \mathcal{T}_0 \to \oO_{C_{ij}}\to \oO_{C_{ij}^{\mathrm{pure}}}\to 0
$$
and the exact sequence
$$
0\to \oO_S(-C^{\mathrm{pure}}_{ij})\to \oO_{S}\to \oO_{C^{\mathrm{pure}}_{ij}}\to 0.
$$
The support of $\mathcal{T}_0$ is 0-dimensional. Applying the Hirzebruch-Riemann-Roch formula gives  
\begin{align*}
\chi\Big(\frac{J_j+I_1^i}{J_j+I_1^{i+1}} \otimes p^* L_2^{-j} \Big)=\chi(\oO_{C_{ij}} \otimes L_1^{-i} \otimes L_2^{-j}) &\geqslant \chi( \oO_{C^{\mathrm{pure}}_{ij}} \otimes L_1^{-i} \otimes L_2^{-j}) \\
&=- \tfrac{1}{2}\beta_{ij}(\beta_{ij} + K_S) - (i L_1 + j L_2) \beta_{ij}.  
\end{align*}
Combining with (\ref{equ sum of chi}), we obtain 
\begin{align} 
\begin{split} \label{key ineq}
\chi(F)\geqslant\chi(\oO_C)&\geqslant - \sum_{j=0}^{\ell} \sum_{i=0}^{\ell_j} ( \tfrac{1}{2} \beta_{ij} (\beta_{ij} + K_S) + (i L_1 + j L_2) \beta_{ij}) \\ 
&   \geqslant  -\tfrac{1}{2} \beta(\beta +K_S) - \beta (L_1 + L_2) \\
&\quad + \tfrac{1}{2} \sum_{((i,j),(i',j')) \atop (i,j) \neq (i',j')} \beta_{ij} \beta_{i'j'} + \beta_{00} (L_1+L_2) + L_1 \sum_{j=1}^{\ell} \beta_{0j} + L_2 \sum_{i=1}^{\ell_0} \beta_{i0},
\end{split}
\end{align}
where we used that  $L_1$ and $L_2$ are nef line bundles. \\

\noindent \textbf{Case 1.} Let $(S,L_1,L_2)=(\mathbb{P}^2,\oO(-1),\oO(-2))$, $\beta = d [H]$ with $d \geqslant 1$, and $n \geqslant 0$. Suppose there exists an element $[(F,s)] \in P_n(X,\beta) \setminus P_n(S,\beta)$. We use the notation above for its scheme theoretic support $C$ and the associated schemes $C_{ij}$. Let $\beta_{ij} = d_{ij} [H]$, then \eqref{key ineq} gives 
\begin{align*}
\chi(F) &\geqslant - \tfrac{1}{2} d^2 + \tfrac{9}{2} d - 3d_{00} - \sum_{j=1}^{\ell} d_{0j} - 2 \sum_{i=1}^{\ell_0} d_{i0} + \tfrac{1}{2} \sum_{((i,j),(i',j')) \atop (i,j) \neq (i',j')} d_{ij} d_{i'j'} \\
&\geqslant - \tfrac{1}{2} d^2 + \tfrac{5}{2} d - d_{00}  + \tfrac{1}{2} \sum_{((i,j),(i',j')) \atop (i,j) \neq (i',j')} d_{ij} d_{i'j'} \\
& \geqslant  - \tfrac{1}{2} d^2 + \tfrac{5}{2} d,
\end{align*}
where the last inequality uses that there exists an $(i,j) \neq (0,0)$ with $d_{ij} \geqslant 1$, because we assumed $C$ is not scheme theoretically supported in the zero section $S \subseteq X$. Hence $n\leqslant -\frac{1}{2}d^2+\frac{5}{2}d-1$ implies
$P_n(S,\beta) \cong P_n(X,\beta)$.
In particular, we find that for the cases (1)--(3) we have $P_n(S,\beta) \cong P_n(X,\beta)$. For case (1) this is obvious from Proposition \ref{proper crit 1} and the fact that $\beta$ is irreducible. 

For $\beta,n$ other than (1)--(3), it is easy to construct a $(\mathbb{C}^*)^4$-fixed stable pair in $$P_n(X,\beta) \setminus P_n(S,\beta)$$ using the combinatorial description of stable pairs in \cite{PT2, CK2}, where $(\mathbb{C}^*)^4$ denotes the torus of the toric Calabi-Yau 4-fold $X$. \\

\noindent \textbf{Case 2.} Let $(S,L_1,L_2)=(\mathbb{P}^1\times\mathbb{P}^1,\oO(-1,-1),\oO(-1,-1))$, $\beta=d_{1}[H_1]+d_{2}[H_2]$ for some $d_1,d_2 \geqslant 0$ not both zero, and $n \geqslant 0$. Suppose there exists an element $[(F,s)] \in P_n(X,\beta) \setminus P_n(S,\beta)$. We use the notation above for its scheme theoretic support $C$ and the associated schemes $C_{ij}$.  Let $\beta_{ij} = d_{1,ij} [H_1] + d_{2,ij} [H_2]$, then \eqref{key ineq} gives
\begin{align}
\begin{split} \label{P1xP1 ineqs}
\chi(F) &\geqslant 3d_1+3d_2-d_1d_2 - 2d_{1,00} - 2d_{2,00} - \sum_{j = 1}^{\ell} (d_{1,0j} + d_{2,0j}) - \sum_{i=1}^{\ell_0} (d_{1,i0} + d_{2,i0}) \\
&\quad + \tfrac{1}{2}  \sum_{((i,j),(i',j')) \atop (i,j) \neq (i',j')}  (d_{1,ij} d_{2,i'j'} + d_{1,i'j'} d_{2,ij}) \\
&\geqslant 2d_1+2d_2-d_1d_2 - d_{1,00} - d_{2,00} + \sum_{((i,j),(i',j')) \atop (i,j) \neq (i',j')}  d_{1,ij} d_{2,i'j'}  \\
&\geqslant d_1 + d_2 - d_1 d_2 + 1 + \sum_{((i,j),(i',j')) \atop (i,j) \neq (i',j')}  d_{1,ij} d_{2,i'j'},
\end{split}
\end{align}
where the last inequality uses that $C$ does not lie scheme theoretically in $S$. Suppose $d_1, d_2 \geqslant 2$, then
\begin{equation} \label{mixineq}
\sum_{((i,j),(i',j')) \atop (i,j) \neq (i',j')}  d_{1,ij} d_{2,i'j'} \geqslant 2.
\end{equation}
Therefore $n \leqslant d_1+d_2 - d_1d_2 + 2$ implies $P_n(S,\beta) \cong P_n(X,\beta)$. In particular, for $\beta,n$ as in (4)--(6), \emph{except} for $(d_1,d_2)=(3,3)$ and $n=0$ (!), we deduce that $P_n(S,\beta) \cong P_n(X,\beta)$. Cases (1)--(3) can be found from \eqref{P1xP1 ineqs} by a similar reasoning. 

For $(d_1,d_2) = (3,3)$, we still use \eqref{P1xP1 ineqs}, but we need to sharpen \eqref{mixineq}. Recall that the schemes $C_j$ and $C_{ij}$ constructed in the first part of the proof are nested.
This implies $d_{1,ij} \geqslant d_{1,i+1j}$ and $d_{2,ij} \geqslant d_{2,i+1j}$ for all $i,j$. Using these inequalities for $(d_1,d_2) = (3,3)$, one can show that 
\begin{equation*} 
\sum_{((i,j),(i',j')) \atop (i,j) \neq (i',j')}  d_{1,ij} d_{2,i'j'}  \geqslant 3.
\end{equation*}
It follows that for $[(F,s)] \in P_n(X,(3,3)) \setminus P_n(S,(3,3))$, we have $\chi(F) \geqslant 1$. Hence $P_0(S,(3,3)) \cong P_0(X,(3,3))$.

For $\beta, n$ other than (1)--(6),  it is easy to construct a $(\mathbb{C}^*)^4$-fixed stable pair in $P_n(X,\beta) \setminus P_n(S,\beta)$ using the combinatorial description of stable pairs in \cite{PT2, CK2}. 
\end{proof}

\begin{rmk} \label{CKK for nef cases}
For $(S,L_1,L_2)$ as in Examples \ref{nef case 1}--\ref{nef case 4}, we found \emph{all} cases for which $n \geqslant 0$, $P_n(X,\beta)$ is proper, and $P_n(S,\beta) \neq \varnothing$ (Propositions \ref{properness prop} and \ref{listproj}). A similar reasoning as in the proof of Proposition \ref{stable pair on del pezzo}  (using \eqref{key ineq}) can be applied to find out when $P_n(X,\beta) \cong P_n(S,\beta)$. In the following cases, we have $n \geqslant 0$, $P_n(X,\beta) \cong P_n(S,\beta)$:
\begin{itemize}
\item $(S,L_1,L_2)=(\mathbb{P}^1\times \mathbb{P}^1, \oO(-1,0), \oO(-1,-2))$, $(d_1,d_2)=(0,1)$ and any $n \geqslant 0$, or $(d_1,d_2)= (0,d)$ for any $d \geqslant 2$ and $n=d$ , or $(d_1,d_2) = (2,2), (2,3), (3,2), (2,4)$ and $n=0$, or $(d_1,d_2) = (2,2), (2,3), (1,d), (d,1)$ for any $d \geqslant 1$ and $n=1$, or $(d_1,d_2) = (1,d)$ for any $d \geqslant 2$ and $n=2$.
\item $(S,L_1,L_2)=(\mathbb{F}_1, \oO(-1,-1), \oO(-1,-2))$, $(d_1,d_2)=(0,1)$ and any $n \geqslant 0$, or $(d_1,d_2)= (0,d)$ for any $d \geqslant 2$ and $n=d$ , or $(d_1,d_2) = (2,3), (2,4), (3,3), (2,5)$ and $n=0$, or $(d_1,d_2) = (2,2), (2,3), (2,4)$ and $n=1$, or $(d_1,d_2) = (1,d)$ for $n=1,2$ and any $d \geqslant n$.
\item $(S,L_1,L_2)=(\mathbb{F}_1, \oO(0,-1), \oO(-2,-2))$, $(d_1,d_2) = (2,3), (2,4), (3,3)$ and $n=0$, or $(d_1,d_2) = (2,2), (2,3), (1,d)$ with $d \geqslant 1$ and $n=1$.
\item $(S,L_1,L_2)=(\mathbb{F}_2, \oO(-1,-2), \oO(-1,-2))$, $(d_1,d_2)=(0,1)$ and any $n \geqslant 0$, or $(d_1,d_2)= (0,d)$ for any $d \geqslant 2$ and $n=d$ , or $(d_1,d_2) = (2,4), (2,5), (2,6)$ and $n=0$, or $(d_1,d_2) = (2,3), (2,4), (2,5)$ and $n=1$, or $(d_1,d_2) = (1,d)$ for $n=1,2$ and any $d \geqslant n$.
\end{itemize}
Furthermore, in all cases listed in Examples \ref{nef case 1}--\ref{nef case 4} but \emph{not} in the above list, one can easily construct a $(\mathbb{C}^*)^4$-fixed stable pair in $P_n(X,\beta) \setminus P_n(S,\beta)$ using the combinatorial description of stable pairs on toric varieties \cite{PT2, CK2}. \\ 
${}$ \\
\textbf{Conclusion.}
For any $(S,L_1,L_2)$ with $L_1 \otimes L_2 \cong K_S$, $L_1^{-1}, L_2^{-1}$ non-trivial and nef, and $S$ minimal and toric, we have
classified \emph{all} $\beta\in H_2(S,\mathbb{Z})$ and $n \geqslant 0$ such that $P_n(X,\beta) \cong P_n(S,\beta) \neq \varnothing$. \\

In the next section, we develop a method to determine the stable pair invariants $P_{n,\beta}([\pt])$ in all of these cases (tabulated in Appendix \ref{tables}).
\end{rmk}

\section{Invariants}

\subsection{Virtual classes of relative Hilbert schemes} \label{sec:relHilb}

For $S$ a smooth projective surface, $\beta \in H_2(S,\mathbb{Z})$, and $n \in \mathbb{Z}$, the moduli space $P_n(S,\beta)$ has a nice description in terms of relative Hilbert schemes due to Pandharipande-Thomas \cite{PT3}. Given a stable pair $[(F,s)] \in P_n(S,\beta)$, one has a short exact sequence
$$
0 \rightarrow \oO_C \rightarrow F \rightarrow Q \rightarrow 0,
$$
where $C$ is the scheme theoretic support of $F$. Dualizing on $C$ yields a short exact sequence
$$
0 \rightarrow F^* \rightarrow \oO_C \rightarrow  \eE xt^1(Q,\oO_C) \rightarrow 0,
$$
where we used $\eE xt^1(F,\oO_C) = 0$ by \cite[Lem.~B.2]{PT3}. Hence $\eE xt^1(Q,\oO_C) \cong \oO_Z$ for some 0-dimensional subscheme $Z \subseteq C$ of length
$$
m = n + g(\beta) - 1 = n + \tfrac{1}{2}\beta(\beta+K_S).
$$
As shown in \cite[Prop. B.8.]{PT3}, the family version of this argument gives an isomorphism
\begin{align} \label{iso rel hilb}
P_{n}(S,\beta)\cong \Hilb^m(\mathcal{C}/H_\beta), 
\end{align}
where $\Hilb^m(\mathcal{C}/H_\beta)$ denotes the relative Hilbert scheme of $m$ points on the fibres of the universal curve
$
\mathcal{C} \rightarrow H_\beta
$
and $H_\beta$ denotes the Hilbert scheme of effective divisors on $S$ in class $\beta$. The description in terms of relative Hilbert schemes helps to establish smoothness. Although we do not need it for this paper, we include the following observation.
\begin{prop} 
In all the cases listed in Proposition  \ref{stable pair on del pezzo} and Remark \ref{CKK for nef cases}, $P_n(S,\beta)$ is smooth.
\end{prop}
\begin{proof}
The method in this proof was also used in \cite{KST}. Let $S = \mathbb{P}^2$, then for any $\beta = d[H]$ with $d \geqslant 1$, and any $n \in \mathbb{Z}$, we have a morphism
\begin{equation} \label{map to Hilb}
P_n(S,\beta) \cong \Hilb^m(\mathcal{C}/|\oO(d)|) \rightarrow S^{[m]},
\end{equation}
where $m = n + \tfrac{1}{2}d(d-3)$. The fibre over $Z \in S^{[m]}$ is the projectivization of the kernel of the evaluation map $H^0(\mathbb{P}^2,\oO(d)) \rightarrow H^0(Z,\oO(d)|_Z)$. It suffices to show that for $n$ and $d \neq 1$ as in Proposition  \ref{stable pair on del pezzo}, this map is surjective. Then it follows that the fibres of \eqref{map to Hilb} are equi-dimensional projective spaces and $P_n(S,\beta)$ is smooth, because $S^{[m]}$ is also smooth. Surjectivity of the evaluation map for all $Z \in S^{[m]}$ is equivalent to $(m-1)$-very ampleness of $\oO(d)$ (by definition, \cite{BS}). Beltrametti-Sommese showed that $\oO(d)$ is $(m-1)$-very ample if and only if $m-1 \leqslant d$, i.e.
$$
n \leqslant d - \tfrac{1}{2}d(d-3) + 1.
$$
This inequality is satisfied for all $n$ and $d \neq 1$ in Cases (2), (3) of Proposition  \ref{stable pair on del pezzo}. Smoothness of $P_n(S,\beta)$ for $d=1$ and any $n$ is clear, because in this case the fibres of \eqref{map to Hilb} are $\mathrm{Sym}^m(\mathbb{P}^1) \cong \mathbb{P}^m$. 

For  $S = \mathbb{P}^1 \times \mathbb{P}^1$, $\oO(d_1 H_1 + d_2 H_2)$ is $k$-very ample if and only if $k \leqslant \mathrm{min}\{d_1,d_2\}$ \cite{BS}. For $S = \mathbb{F}_a$ (for any $a \geqslant 1$), $\oO(d_1 B + d_2 F)$ is $k$-very ample if and only if $k \leqslant \mathrm{min}\{d_1,d_2 - ad_1\}$ \cite{BS}. The proof in the remaining cases of Proposition \ref{stable pair on del pezzo} and Remark \ref{CKK for nef cases} then follows similarly.
\end{proof}
Let $\mathcal{Z}\subseteq S\times S^{[m]}$ be the universal  subscheme and denote the pull-back of 
$\mathcal{Z}$ to $S\times S^{[m]}\times H_\beta$ by the same symbol (and similarly for $\mathcal{C}\subseteq S\times H_\beta$).
Consider the rank $m$ vector bundle 
\begin{align*}
\oO(\mathcal{C})^{[m]}:=\pi_*(\oO(\mathcal{C})|_{\mathcal{Z}}) 
\end{align*}
on $S^{[m]}\times H_\beta$, where $\pi: S\times S^{[m]}\times H_\beta \to S^{[m]}\times H_\beta$ is the projection. By \cite[App.~A]{KT1}, there exists a tautological section $s$ of $\oO(\mathcal{C})^{[m]}$ cutting out $\Hilb^m(\mathcal{C}/H_\beta)$ from its ambient space
\begin{align*}\xymatrix{& \oO(\mathcal{C})^{[m]} \ar[d]_{ }  &  &  \\
  s^{-1}(0)\cong \Hilb^m(\mathcal{C}/H_\beta) \ar@{^(->}[r]^{\quad \quad \jmath} & S^{[m]}\times H_\beta.  \ar@/_1pc/[u]_{s} &  &  }
\end{align*}
In general $H_\beta$ is not smooth (or smooth but not of expected dimension). Therefore, this construction only provides a \emph{relative} perfect obstruction theory on $\Hilb^m(\mathcal{C}/H_\beta) \rightarrow H_\beta$. The Hilbert scheme of divisors $H_\beta$ has a natural perfect obstruction theory
$$
(\dR p_* \oO_{\mathcal{C}}(\mathcal{C}))^\vee \rightarrow \mathbb{L}_{H_\beta},
$$
where $p : S \times H_\beta \to H_\beta$ denotes projection. This is the perfect obstruction theory used to define the Poincar\'e/Seiberg-Witten invariants of $S$ in \cite{DKO, CK}. Taken together, these provide an \emph{absolute} perfect obstruction theory on $P_n(S,\beta) \cong \Hilb^m(\mathcal{C}/H_\beta)$ by \cite[App.~A.3]{KT1}. The virtual tangent bundle of this absolute perfect obstruction theory is
$$
\dR \hH om_{\pi_S}(\mathbb{I}_S^\mdot, \mathbb{F})
$$
where $\mathbb{I}_S^\mdot = \{\oO \to \mathbb{F}\}$ denotes the universal stable pair on $S \times P_n(S,\beta)$ and $\pi_S : S \times P_n(S,\beta) \to P_n(S,\beta)$ is the projection. By \cite[Prop.~2.1]{Koo}, the resulting virtual class satisfies
\begin{align} \label{virt class rel Hilb} 
\jmath_*[\Hilb^m(\mathcal{C}/H_\beta)]^{\mathrm{vir}}= (S^{[m]}\times [H_\beta]^{\mathrm{vir}}) \cdot e\big(\oO(\mathcal{C})^{[m]}\big). 
\end{align}
The corresponding virtual class on $P_n(S,\beta)$, via the isomorphism \eqref{iso rel hilb}, is denoted by $[P_n(S,\beta)]^{\mathrm{vir}}$.

\subsection{Comparison of virtual classes} \label{sec:comparison}

Let $S$ be a smooth projective surface and $L_1, L_2 \in \Pic(S)$ such that $L_1 \otimes L_2 \cong K_S$. We consider the local surface $X = \mathrm{Tot}_S(L_1 \oplus L_2)$, which is a Calabi-Yau 4-fold. 
Fix $n \in \mathbb{Z}$ and $\beta \in H_2(S,\mathbb{Z})$ such that $P_n(X,\beta)$ is proper. Then it has a virtual class
\begin{equation} \label{virtual class}
[P_n(X, \beta)]^{\rm{vir}}\in H_{2n}\big(P_n(X, \beta),\mathbb{Z}\big),
\end{equation}
in the sense of Borisov-Joyce \cite{BJ}, which depends on a choice of orientation on $P_n(X,\beta)$. In order to apply \cite{BJ} to our case, we need  the existence of a $(-2)$-shifted symplectic structure on $P_n(X,\beta)$ and an orientability result as reviewed in Section \ref{review dt4}.
The existence of a $(-2)$-shifted symplectic structure on $P_n(X,\beta)$ was shown in \cite[Thm.~7.3.2]{Bus} and \cite[Thm.~4.0.8]{Pre}. By \cite[Thm.~5.3]{CMT2}, the existence of orientations on $P_n(X,\beta)$ can be reduced to the existence of orientations on the moduli stack $\mathfrak{M}_n(X, \beta)$ of coherent sheaves $F$ on $X$ with 1-dimensional proper support of class $\beta$ and $\chi(F) = n$. 
Taking $Y = \mathrm{Tot}_S(L_1)$, we have $X = \mathrm{Tot}_Y(K_Y)$. By considering the derived enhancement of $\mathfrak{M}_n(X, \beta)$, it is the $(-2)$-shifted cotangent bundle of a derived moduli stack of sheaves on $Y$. Therefore $\mathfrak{M}_n(X, \beta)$ has an orientation 
(see e.g. \cite[Lem.~4.3]{Toda} for a similar argument in the case of Calabi-Yau threefolds).

We denote by $[\mathrm{pt}] \in H^4(X,\mathbb{Z})$ the pull-back along $\pi : X \rightarrow S$ of the Poincar\'e dual of the point class on $S$. Using the same notation as in Section \ref{sect on pair inv on cpt CY4}, we define stable pair invariants
\begin{align}\label{glob inv}
P_{n,\beta}([\pt]):=&\int_{[P_n(X,\beta)]^{\rm{vir}}}\tau([\pt])^n \in \mathbb{Z}. 
\end{align}
When $n=0$, we simply write $P_{0,\beta}:=P_{0,\beta}([\pt])$.
  
Assuming $P_n(X,\beta) \cong P_n(S,\beta)$, we can compare the virtual class \eqref{virtual class} to the virtual class on the relative Hilbert scheme \eqref{virt class rel Hilb} studied in \cite{KT1, KT2}. In Proposition \ref{stable pair on del pezzo} and Remark \ref{CKK for nef cases} we gave a list of examples where this assumption is satisfied.
\begin{prop}\label{compare virtual class}
Let $S$ be a smooth projective surface, $L_1, L_2 \in \Pic(S)$ such that $L_1 \otimes L_2 \cong K_S$ and let $X = \mathrm{Tot}_S(L_1 \oplus L_2)$. Suppose $\beta \in H_2(S,\mathbb{Z})$ and $n \geqslant 0$ are chosen such that $P_n(X,\beta) \cong P_n(S,\beta)$. Then there exists a choice of orientation such that 
\begin{align*}
[P_n(X,\beta)]^{\mathrm{vir}}=(-1)^{\beta \cdot L_2 +n} \cdot e\big(-\dR\hH om_{\pi_{P_{S}}}(\mathbb{F}, \mathbb{F} \boxtimes L_1)\big) \cdot [P_n(S,\beta)]^{\mathrm{vir}}.   
\end{align*}
Here $[P_n(S,\beta)]^{\mathrm{vir}}$ is the virtual class induced from the relative Hilbert scheme (Section \ref{sec:relHilb}), $\mathbb{I}_S^{\mdot}=\{\oO \to \mathbb{F}\}$ denotes the universal stable pair on $S \times P_n(S, \beta)$, and $\pi_{P_{S}} \colon S \times P_n(S, \beta) \to P_n(S, \beta)$ is the projection.
The sign results from a preferred choice of orientation. 
\end{prop}
\begin{proof}
Let $Y=\mathrm{Tot}_S(L_1)$. Then $X=\mathrm{Tot}_Y(K_Y)$ is the total space of the canonical bundle of $Y$. 
By the assumption, we have isomorphisms of moduli spaces
\begin{align}\label{isom of mod}
P_n(S,\beta) \cong P_n(Y,\beta) \cong P_n(X,\beta). 
\end{align}
Let $\iota : S \hookrightarrow Y$ denote the zero section. A stable pair $I_S^\mdot =\{s : \oO_S \to F\}\in P_n(S,\beta)$ on $S$ induces a stable pair
$$
I_Y^\mdot = \{\oO_Y\to \iota_* \oO_S \stackrel{\iota_*s}{\to} \iota_*F\}
$$ 
on $Y$. Consider the distinguished triangle
\begin{align}\label{triangle start}
I_Y^\mdot \to \oO_Y \to \iota_{\ast}F.
\end{align} 
Applying $\RHom_Y(I_Y^\mdot, \cdot)$ and taking out trace gives a distinguished triangle
\begin{align*} 
\RHom_Y(I_Y^\mdot,\iota_*F) \to \RHom_Y(I_Y^\mdot,I_Y^\mdot)_0[1]\to \RHom_Y(\iota_*F,\oO_Y)[2].
\end{align*}
Applying adjunction and the isomorphism 
\begin{equation} \label{der pull-back}
\dL \iota^*I_Y^\mdot\cong I_S^\mdot \oplus F\otimes L_1^{-1}
\end{equation}
gives a long exact sequence 
$$
\cdots \to \Ext^i_S(I_S^\mdot,F)\oplus \Ext^i_S(F,F\otimes L_1)\to \Ext^{i+1}_Y(I_Y^\mdot,I_Y^\mdot)_0 \to \Ext^{i+2}_Y(\iota_*F,\oO_Y) \to \cdots. 
$$
Note that $\Ext^{1}_Y(\iota_*F,\oO_Y) \cong \Ext^2_Y(\oO_Y,\iota_* F \otimes K_Y)^\vee = 0$. Furthermore, the isomorphism \eqref{isom of mod} induces an isomorphism on Zariski tangent spaces
$$
\Ext^0_S(I_S^\mdot, F)\cong \Ext^1_Y(I_Y^\mdot, I_Y^\mdot)_0. 
$$
Therefore, we deduce $\Hom_S(F,F\otimes L_1)=0$ (similarly $\Hom_S(F,F\otimes L_2)=0$). This vanishing allows us to conclude that the natural (Le Potier) pair obstruction theory 
\begin{equation} \label{pairpot}
(\dR \hH om_{\pi_{P_Y}}(\mathbb{I}_Y^\mdot, \iota_{P_Y *} \mathbb{F}))^{\vee} \to \mathbb{L}_{P_n(Y,\beta)}
\end{equation}
is \emph{perfect}, i.e.~2-term, as we will now show\,\footnote{This was proved for irreducible $\beta$ in \cite[Lem.~3.1]{CMT2}.}. 
Here $\mathbb{I}_Y^\mdot = \{\oO \to \iota_{P_Y *} \mathbb{F} \}$ denotes the universal stable pair on $Y \times P_n(Y,\beta)$, $\iota_{P_Y} : S \times P_n(Y,\beta) \hookrightarrow Y \times P_n(Y,\beta)$ is the base change of the zero section, and $\pi_{P_Y} : Y \times P_n(Y,\beta) \to P_n(Y,\beta)$ denotes the projection. 

From the distinguished triangle
\begin{equation} \label{ex triangle 2}
\RHom_Y(\iota_* F, \iota_*F) \to \RHom_Y(\oO_Y,\iota_* F) \to \RHom_Y(I_Y^{\mdot}, \iota_* F),
\end{equation}
we obtain an exact sequence
$$
0 = H^2(Y,\iota_* F) \to \Ext^2_Y(I_{Y}^\mdot,\iota_* F) \to \Ext^3_Y(\iota_*F,\iota_*F) \to 0 \to \Ext^3_Y(I_Y^\mdot, \iota_* F) \to 0.
$$
Moreover, by adjunction and $\dL \iota^* F \cong F \oplus F \otimes L_1^{-1}[1]$, we have
\begin{align*}
\Ext^3_Y(\iota_*F,\iota_*F)&\cong \Ext^3_S(F,F)\oplus \Ext^2_S(F,F\otimes L_1)\cong \Hom_S(F,F\otimes L_2)^{\vee}=0. 
\end{align*}
Hence $\Ext^2_Y(I_{Y}^\mdot,\iota_* F) \cong \Ext^3_Y(\iota_* F,\iota_* F) =0$. Also note that $\Hom_Y(\iota_*F,\iota_*F) \to \Hom_Y(\oO_Y,\iota_* F)$ is injective. Therefore $\Ext^i_Y(I_Y^{\mdot},\iota_* F) = 0$ unless $i=0,1$ and the complex \eqref{pairpot} is 2-term. We denote the corresponding virtual class by $[P_n(Y,\beta)]^{\mathrm{vir}}_{\mathrm{pair}}$.

We can now use the argument of  \cite[Prop.~4.3]{CMT2} to deduce that the 4-fold virtual class $[P_n(X,\beta)]^{\mathrm{vir}}$ of \eqref{virtual class} equals the pairs virtual class $[P_n(Y,\beta)]^{\mathrm{vir}}_{\mathrm{pair}}$. For completeness, we repeat the argument. Just like pushing forward from $S$ to $Y$ gives \eqref{triangle start} and \eqref{der pull-back}, pushing forward further to $X$ gives
\begin{align*}
\RHom_X(I_X^\mdot,\jmath_* \iota_* F) \to &\RHom_X(I_X^\mdot,I_X^\mdot)_0[1]\to \RHom_X(\jmath_* \iota_*F,\oO_X)[2], \\
&\dL \jmath^*I_X^\mdot\cong I_Y^\mdot \oplus \iota_* F\otimes K_Y^{-1}, 
\end{align*}
where 
$$
I_X^\mdot = \{\oO_X\to \jmath_* \iota_* \oO_S \stackrel{\jmath_* \iota_* s}{\to} \jmath_* \iota_* F\}
$$ 
and we denote the zero section by $\jmath : Y \hookrightarrow X$. Let $T$ be the cone of the composition
$$
\RHom_Y(I_Y^\mdot,\iota_*F) \to \RHom_X(I_X^\mdot, \jmath_* \iota_*F) \to  \RHom_X(I_X^\mdot,I_X^\mdot)_0[1].
$$
Then $T$ fits in the distinguished triangles
\begin{align}  
\begin{split} \label{ex triangles T}
&\RHom_Y(I_Y^\mdot, \iota_* F) \to \RHom_X(I_X^\mdot,I_X^\mdot)_0[1] \to T, \\
&\RHom_Y(\iota_* F, \iota_*F \otimes K_Y) \to T \to \RHom_X(\jmath_* \iota_* F, \oO_X)[2].
\end{split}
\end{align}
Applying Serre duality to the first and third term of the second distinguished triangle, dualizing, and shifting gives the following distinguished triangle
$$
\RHom_Y(\iota_*F, \iota_*F)[2] \to \RHom_X(\oO_X,\jmath_*\iota_*F)[2] \to T^{\vee}.
$$
Comparing to \eqref{ex triangle 2}, we obtain $T \cong \RHom_Y(I_Y^\mdot, \iota_*F)^\vee[-2]$. Hence from \eqref{ex triangles T} we get a short exact sequence
$$
0 \to \Ext^1_Y(I_Y^\mdot, \iota_*F) \to \Ext^2_X(I_X^\mdot, I_X^\mdot)_0 \to \Ext^1_Y(I_Y^\mdot, \iota_*F)^\vee \to 0,
$$
where we crucially used $\Ext^2_Y(I_Y^\mdot,\iota_* F) = 0$ which was shown above. This way, we obtain a half-dimensional subspace $\Ext^1_Y(I_Y^\mdot, \iota_* F)$ of $ \Ext^2_X(I_X^\mdot, I_X^\mdot)_0$. One can show that it is isotropic by the exact same argument as in the proof of \cite[Prop.~3.3, Prop.~2.11]{CMT2}. From this, it is concluded in loc.~cit.~that 
$$
[P_n(X,\beta)]^{\mathrm{vir}} = (-1)^{\beta \cdot L_2 +n}\cdot [P_n(Y,\beta)]^{\mathrm{vir}}_{\mathrm{pair}}.
$$
Here the sign comes from a choice of preferred orientation discussed in a similar setting in \cite{CaoFano}.

Finally, we express the pairs virtual class on $Y$ in terms of the pairs virtual class on $S$. By adjunction, we have 
\begin{align*}
\dR \hH om_{\pi_{P_Y}}(\mathbb{I}_Y^\mdot,\iota_{P_Y *} \mathbb{F})&\cong \dR \hH om_{\pi_{P_S}}(\dL \iota_{P_Y}^* \mathbb{I}_Y^\mdot,\mathbb{F})\\
&\cong \dR \hH om_{\pi_{P_S}}(\mathbb{I}_S^\mdot,\mathbb{F}) \oplus \dR \hH om_{\pi_{P_S}}(\mathbb{F},\mathbb{F} \boxtimes L_1),   
\end{align*}
where $\pi_{P_S} : S \times P_n(S,\beta) \rightarrow P_n(S,\beta) \cong P_n(Y,\beta)$ denotes the projection. From the vanishing $\Hom_S(F,F\otimes L_1) = \Hom_S(F,F\otimes L_2) =0$, for all $[(F,s)] \in P_n(S,\beta)$, we deduce that 
$$
- \dR \hH om_{\pi_{P_S}}(\mathbb{F},\mathbb{F} \boxtimes L_1) \cong \eE xt^1_{\pi_{P_S}} (\mathbb{F},\mathbb{F} \boxtimes L_1)
$$ 
is locally free on $P_n(Y,\beta) \cong P_n(S,\beta)$. Hence the two virtual tangent bundles on $P_n(Y,\beta) \cong P_n(S,\beta)$ differ by a locally free sheaf (in degree $1$). Therefore, by \cite[Thm.~4.6]{Sie}, we have
\begin{equation*}
[P_n(Y,\beta)]^{\mathrm{vir}}_{\mathrm{pair}} = e\big(-\dR\hH om_{\pi_{P_{S}}}(\mathbb{F}, \mathbb{F} \boxtimes L_1)\big) \cdot [P_n(S,\beta)]^{\mathrm{vir}}. \qedhere
\end{equation*}
\end{proof}

\begin{rmk} 
Let $S$ be a smooth projective surface satisfying $b_1(S) = p_g(S) = 0$. Let $L_1, L_2 \in \Pic(S)$ such that $L_1 \otimes L_2 \cong K_S$ and $X = \mathrm{Tot}_S(L_1 \oplus L_2)$. Suppose $P_n(X,\beta)$ is proper. Once the virtual localization formula for stable pair theory on Calabi-Yau 4-folds is established\,\footnote{In the special case $S$ is moreover toric and $P_n(S,\beta) \cong P_n(X,\beta)$ is smooth, a virtual localization formula was proved in \cite[Thm.~A.1]{CK2}. A general virtual localization formula has been announced by Oh-Thomas \cite{OT}.}, we expect that it will induce a virtual class on each of the connected components of $P_n(X,\beta)^{\mathbb{C}^*}$, where $\mathbb{C}^*$ is the 1-dimensional subtorus, preserving the Calabi-Yau volume form, inside the torus $\mathbb{C}^* \times \mathbb{C}^*$ acting on the fibres of $X$. When $P_n(S,\beta) \subseteq P_n(X,\beta)^{\mathbb{C}^*}$ is open and closed, this gives a virtual class, which we expect to be given by 
(for an appropriate choice of orientation) 
\begin{equation} \label{expect}
(-1)^{\beta \cdot L_2 +n} \cdot e\big(-\dR\hH om_{\pi_{P_{S}}}(\mathbb{F}, \mathbb{F} \boxtimes L_1) \otimes \mathfrak{t}_1 \big) \cdot [P_n(S,\beta)]^{\mathrm{vir}} ,  
\end{equation}
where $[P_n(S,\beta)]^{\mathrm{vir}}$ is the virtual class induced from the relative Hilbert scheme (Section \ref{sec:relHilb}), $\mathfrak{t}_1$ is a primitive character corresponding to the first component of the action of $\mathbb{C}^* \times \mathbb{C}^*$, and $e(\cdot)$ denotes equivariant Euler class. 
When $P_n(X,\beta)$ is non-proper, one could define the contribution of $P_n(S,\beta)$ to the stable pair invariants of $X$ by \eqref{expect} (capped with appropriate insertions). 
\end{rmk}

\subsection{Main theorem}

We are now ready to prove the theorem of the introduction. Recall from \eqref{twisted tangent bundle} that we denote by $T_{S^{[m]}}(\lL)$ the twisted (by $\lL$) tangent bundle of $S^{[m]}$.
\begin{thm} \label{mainthm} 
Let $S$ be a smooth projective surface with $b_1(S) = p_g(S) = 0$ and $L_1, L_2 \in \Pic(S)$ such that $L_1 \otimes L_2 \cong K_S$. 
Suppose $\beta \in H_2(S,\mathbb{Z})$ and $n \geqslant 0$ are chosen such that $P_n(X,\beta) \cong P_n(S,\beta)$ for $X = \mathrm{Tot}_S(L_1 \oplus L_2)$. 
Denote by $[\mathrm{pt}] \in H^4(X,\mathbb{Z})$ the pull-back of the Poincar\'e dual of the point class on $S$.
Let $P_n(X,\beta)$ be endowed with the orientation as in \eqref{intro choice of ori}. Then
\begin{align*}
P_{n,\beta}([\mathrm{pt}])&=(-1)^{\beta \cdot L_2+n} \int_{S^{[m]}\times \mathbb{P}^{\chi(\beta)-1}} c_{m}(\oO_S(\beta)^{[m]}(1)) \, \frac{h^n(1+h)^{\chi(L_1(\beta))} (1-h)^{\chi(L_2(\beta))} \, c(T_{S^{[m]}}(L_1))}{c(L_1(\beta)^{[m]} (1))\cdot c((L_2(\beta)^{[m]}(1))^{\vee})},
\end{align*}
when $\beta^2 \geqslant 0$. Here $m := n + g(\beta) - 1$ and $h := c_1(\oO(1))$. Moreover, $P_{n,\beta}([\mathrm{pt}]) = 0$ when $\beta^2<0$.
\end{thm}
\begin{proof}
Suppose $\beta$ is an effective divisor and $m \geqslant 0$, otherwise $P_n(S,\beta) \cong \Hilb^m(\mathcal{C}/H_\beta) = \varnothing$ and $P_{n,\beta}([\mathrm{pt}])=0$. Consider the closed embedding 
$$
\jmath : \Hilb^m(\mathcal{C}/H_\beta) \hookrightarrow S^{[m]}\times H_\beta,
$$
as in Section \ref{sec:relHilb}. Below, we will show that there exists a class $\psi \in K_0(S^{[m]} \times H_\beta)$ restricting to $e(-\dR\hH om_{\pi_{P_{S}}}(\mathbb{F}, \mathbb{F} \boxtimes L_1)) \cdot \tau([\pt])^n$ on $\Hilb^m(\mathcal{C}/H_\beta)$. 
By Proposition \ref{compare virtual class}, it follows that
\begin{equation} \label{push to ambient}
P_{n,\beta}([\pt]) =\int_{S^{[m]}\times [H_\beta]^{\mathrm{vir}}} c_{m}(\oO(\mathcal{C})^{[m]}) \cdot \psi.
\end{equation}
Since $b_1(S) = p_g(S)=0$, we have \cite{DKO, KT2}
$$
[H_\beta]^{\mathrm{vir}} = |\beta|^{\mathrm{vir}} = h^{h^1(\oO(\beta))} \cap |\beta| \in H_{2\chi(\beta)-2}(|\beta|),
$$
where $h$ denotes the class of the hyperplane on $H_\beta = |\beta|$. Furthermore, we have 
$$
\oO(\mathcal{C})^{[m]}:=\pi_*(\oO(\mathcal{C})|_{\mathcal{Z}}) \cong \oO(\beta)^{[m]}(1),
$$
which follows from the isomorphism $\oO(\mathcal{C}) \cong \oO_S(\beta) \boxtimes \oO(1)$ on $S \times |\beta|$. 
Therefore $P_{n,\beta}([\pt]) = 0$, unless $\chi(\beta) \geqslant 1$, which we assume from now on.

Recall from the proof of Proposition \ref{compare virtual class} that $- \dR \hH om_{\pi_{P_S}}(\mathbb{F},\mathbb{F} \boxtimes L_1) \cong \eE xt^1_{\pi_{P_S}} (\mathbb{F},\mathbb{F} \boxtimes L_1)$ is locally free on $P_n(S,\beta)$ and its rank is $\beta^2$. Therefore $P_{n,\beta}([\pt]) = 0$ unless $\beta^2 \geqslant 0$, which we assume from now on. Next, we extend the complex $-\dR\hH om_{\pi_{P_{S}}}(\mathbb{F}, \mathbb{F} \boxtimes L_1)$ from $P_{n}(S,\beta)$ to $S^{[m]}\times |\beta|$. In $K$-theory, we have 
$\mathbb{I}_S^\mdot=\oO - \mathbb{F}$ and 
\begin{align*}
-\dR\hH om_{\pi_{P_{S}}}(\mathbb{F}, \mathbb{F} \boxtimes L_1)
&= -\chi(L_1) \otimes \oO
+\dR\hH om_{\pi_{P_{S}}}(\oO, \mathbb{I}_S^\mdot \boxtimes L_1) \\
&\quad +\dR\hH om_{\pi_{P_{S}}}(\mathbb{I}_S^\mdot, L_1)-\dR\hH om_{\pi_{P_{S}}}(\mathbb{I}_S^\mdot, \mathbb{I}_S^\mdot \boxtimes L_1) \\
&=-\chi(L_1) \otimes \oO +\dR\hH om_{\pi_{P_{S}}}((\mathbb{I}_S^\mdot)^{\vee}, L_1) \\
&\quad +\dR\hH om_{\pi_{P_{S}}}(\oO,(\mathbb{I}_S^\mdot)^{\vee}\boxtimes L_1)-\dR\hH om_{\pi_{P_{S}}}((\mathbb{I}_S^\mdot)^{\vee}, (\mathbb{I}_S^\mdot)^{\vee} \boxtimes L_1)\big),
\end{align*}
where we suppressed some obvious pull-backs. On $S \times |\beta| \times S^{[m]}$, we have the sheaf  $\mathcal{I} \boxtimes \oO_S(\beta) \boxtimes \oO(1)$, where we use the notation from the introduction. By \cite[Lem.~A.4]{KT1}, we have
\begin{align*}
(\mathbb{I}_S^\mdot)^{\vee} \cong \mathcal{I} \boxtimes \oO_S(\beta) \boxtimes \oO(1) |_{\Hilb^m(\mathcal{C} / |\beta|) \times S}.
\end{align*}
Next we can replace $\mathcal{I}$ by $\oO - \oO_{\mathcal{Z}}$ in $K$-theory. 
Then $-\dR\hH om_{\pi_{P_{S}}}(\mathbb{F}, \mathbb{F} \boxtimes L_1)$ is the restriction of the following element
in the $K$-group of $S^{[m]} \times |\beta|$
\begin{align}
\begin{split} \label{K group class}
&-\chi(L_1)\otimes \oO+ \dR\hH om_{\pi}\big((\oO-\oO_{\mathcal{Z}})\boxtimes (\oO_S(\beta)\otimes L_1^{-1}) \boxtimes \oO(1),\oO\big) \\
&\quad +\dR\hH om_{\pi}\big(\oO,(\oO-\oO_{\mathcal{Z}})\boxtimes (\oO_S(\beta)\otimes L_1) \boxtimes \oO(1)\big)-
\dR\hH om_{\pi}(\mathcal{I}, \mathcal{I} \boxtimes L_1) \\
&=-\chi(L_1)\otimes \oO+\chi(L_1(\beta))\otimes \oO(1)+\chi(L_2(\beta))\otimes \oO(-1) \\
&\quad -(L_1(\beta))^{[m]}\boxtimes \oO(1)-\big((L_2(\beta))^{[m]}\big)^{\vee}\boxtimes \oO(-1) -
\dR\hH om_{\pi}(\mathcal{I}, \mathcal{I} \boxtimes L_1),
\end{split}
\end{align}
where $\pi : S \times  S^{[m]} \times |\beta| \to S^{[m]} \times |\beta|$ denotes the projection and we used Serre duality, $L_1 \otimes L_2 \cong K_S$, and $L_i(\beta):=\oO_S(\beta)\otimes L_i$. 

Finally, we consider primary insertions
\begin{align*}
\tau \colon H^{4}(X,\mathbb{Z})\to H^{2}(P_n(X,\beta),\mathbb{Z}), \quad
\tau(\gamma)=\pi_{P\ast}(\pi_X^{\ast}\gamma \cup\ch_3(\iota_*\mathbb{F}) ),
\end{align*}
where $\pi_X$, $\pi_P$ are projections from $X \times P_n(X,\beta)$ to corresponding factors, 
and $\iota: S\times P_n(S,\beta)\hookrightarrow X\times P_n(S,\beta) \cong X \times P_n(X,\beta)$ is the base change of the inclusion of the zero section. Note that $\ch_3(\iota_*\mathbb{F})$ is Poincar\'e dual to the fundamental class of the scheme theoretic support of $\iota_*\mathbb{F}$, which we denote by $[\iota_* \mathbb{F}]$. The fundamental class of the scheme theoretic support of $\mathbb{F}$, which we denote by $[\mathbb{F}]$, equals $\jmath^* [\mathcal{C}]$ where $\mathcal{C} \subseteq S \times S^{[m]} \times |\beta|$ is the pullback of the universal curve over $|\beta|$.
Consider the commutative diagram 
\begin{align*}\xymatrix{ X  \ar[r]^{p}  & S \ar@/_1pc/[l]_{\iota}  & \\ 
X \times P_n(X,\beta)  \ar[u]_{\pi_X}\ar[d]_{\pi_P} \ar[r]^{p}  & S \times P_n(S,\beta) \ar@/_1pc/[l]_{\iota} \ar[u]_{\pi_S} \ar[d]_{\pi_P} 
\ar[r]^{\jmath \,\,\, } & S \times S^{[m]} \times |\beta| \ar[d]_{\pi} \ar[ul]_{\pi_S} \\
P_n(X,\beta) \ar[r]^{\cong} & P_n(S,\beta) \ar[r]^{\jmath\,\, } & S^{[m]}\times |\beta|  }
\end{align*}
For $\gamma := p^*[\mathrm{pt}] \in H^{4}(X,\mathbb{Z})$, where $[\mathrm{pt}]$ denotes the Poincar\'e dual of the point class on $S$, we work our way through the diagram
\begin{align}
\begin{split} \label{tau}
\tau(\gamma)&=\pi_{P\ast}(\pi_X^{\ast}\gamma \cup[\iota_*\mathbb{F}]) \\
&= \pi_{P\ast}\,p_{\ast}(p^*\pi_S^{\ast} [\pt] \cup \iota_* [\mathbb{F}])\\
&= \pi_{P\ast}(\pi_S^{\ast} [\pt] \cup [\mathbb{F}]) \\
&= \jmath^*\big(\pi_{\ast}(\pi_S^{\ast} [\pt] \cup [\mathcal{C}])\big).
\end{split}
\end{align}
Using, once more, that on $S \times |\beta|$ we have $\oO(\mathcal{C}) \cong \oO_S(\beta) \boxtimes \oO(1)$, we conclude
$$
\pi_{\ast}(\pi_S^{\ast}[\pt] \cup [\mathcal{C}])= h \int_{S} \pt = h.
$$
Therefore $\tau([\pt])^n$ is simply the restriction of the class $h^n$ on $S^{[m]} \times |\beta|$.

Since $\rk(- \dR \hH om_{\pi_{P_S}}(\mathbb{F},\mathbb{F} \boxtimes L_1)) = \beta^2 = 2m + \chi(\beta) - 1 - m - n$, we can replace Euler class by total Chern class. The result now follows from \eqref{push to ambient}, \eqref{K group class}, and \eqref{tau}.
\end{proof}
\begin{rmk}\label{rmk on b1>0}
For surfaces with $p_g(S) = 0$ and $b_1(S) > 0$, we can still use the formula for the virtual class from Proposition \ref{compare virtual class}. Suppose $h^2(L) = 0$ for all $L \in \Pic_\beta(S)$. Then the virtual class $[H_\beta]^{\mathrm{vir}}$ can be calculated by fixing a sufficiently ample effective divisor $A$ on $S$ and considering the embedding $H_\beta \hookrightarrow H_{[A] + \beta}$ as in \cite{DKO} (see also \cite[Prop.~A.2]{KT1}, \cite{KT2}). Therefore the invariant can be expressed as an integral over $S^{[m]} \times H_{[A] + \beta}$, where $H_{[A]+\beta}$ is a projective bundle over $\Pic_{[A]+\beta}(S)$ via the Abel-Jacobi map. Pushing forward along the Abel-Jacobi map, the invariant can be expressed as an integral over $S^{[m]} \times \Pic_{A + \beta}(S)$.
\end{rmk}

\subsection{Atiyah-Bott localization} \label{sec:torusloc}

In Corollary \ref{stable pair on del pezzo}, Remark \ref{CKK for nef cases}, we gave examples of $(S,L_1,L_2)$ for which the assumptions of Theorem \ref{mainthm} are satisfied. In all of these cases, $S$ is a toric surface. As a consequence, $X = \mathrm{Tot}_S(L_1 \oplus L_2)$ is also toric, so in principle one could calculate the invariant $P_{n,\beta}([\mathrm{pt}])$ using the vertex formalism for stable pair invariants on toric Calabi-Yau 4-folds developed in \cite{CK2, CKM}\,\footnote{In loc.~cit.~it is assumed that the fixed locus $P_n(X,\beta)^{(\mathbb{C}^*)^4}$ is at most 0-dimensional. This is the case for all local Calabi-Yau 4-fold surfaces.}. However, the number of $(\mathbb{C}^*)^4$-fixed points is typically very large. For instance, for $(S,L_1,L_2) = (\mathbb{P}^1 \times \mathbb{P}^1, \oO(-1,-1), \oO(-1,-1))$, $(d_1,d_2) = (2,4)$, and $n=0$, we have 182 fixed points, whereas Theorem \ref{mainthm} only involves an integral over $S^{[2]} \times \mathbb{P}^{14}$.  

The calculation of intersection numbers on Hilbert schemes of points on toric surfaces is a classical subject (see e.g.~\cite{ES}). 
Let $S$ be a smooth projective toric surface with torus $T = (\mathbb{C}^*)^2$. The action of $T$ on $S$ lifts to an action of $T$ on $S^{[m]}$ for any $m$. Let $P$ be a polynomial expression in Chern classes of
\begin{equation} \label{taut complexes}
\dR \Gamma(\lL) \otimes \oO  - \dR \hH om_{\pi}(\iI,\iI \boxtimes \mathcal{L}), \quad \lL^{[m]},
\end{equation}
for various choices of $T$-equivariant line bundles $\lL$ on $S$ and where $\pi : S \times S^{[m]} \to S^{[m]}$ denotes the projection. Note that this includes Chern classes of the tangent bundle, which can be expressed as $-\dR \hH om_{\pi}(\iI,\iI)_0$ (since $\Ext^{i}(I_Z,I_Z)_0 = 0$ for $Z \in S^{[m]}$ and $i \neq 1$). Suppose also that the degree of $P$, as a class in the Chow ring $A^*(S^{[m]})$, equals $\dim S^{[m]} = 2m$.
By the Atiyah-Bott localization formula \cite{AB}, we have
$$
\int_{S^{[m]}} P = \int_{(S^{[m]})^T} \frac{P |_{(S^{[m]})^T}}{e(N_{(S^{[m]})^T / S^{[m]}})},
$$
where $e(\cdot)$ denotes the $T$-equivariant Euler class and $N_{(S^{[m]})^T / S^{[m]}}$ is the normal bundle of the fixed point locus $(S^{[m]})^T  \subseteq S^{[m]}$. Furthermore, in this formula one has to choose a $T$-equivariant lift of $P$. More precisely, one can choose a $T$-equivariant structure on all (complexes of) sheaves appearing in $P$ and replace all Chern classes appearing in $P$ by $T$-equivariant Chern classes.

The fixed point locus consists of isolated reduced points, which can be described combinatorially. 
Consider a cover by maximal $T$-invariant affine open subsets:
$$
\big\{U_{\sigma} \cong \Spec \mathbb{C}[x_{\sigma}, y_{\sigma} ]\big\}_{\sigma=1}^{e(S)}.
$$
Then the fixed locus $(S^{[m]})^T$ precisely consists of the closed subschemes of $S$ defined by collections of monomial ideals
$$
\big\{ I_\sigma \subseteq \mathbb{C}[x_{\sigma}, y_{\sigma}] \big\}_{\sigma=1}^{e(S)}
$$ 
of total colength $m$. The monomial ideals of finite colength in $\mathbb{C}[x, y]$ are in bijective correspondence with partitions. Explicitly, $\lambda = (\lambda_1 \geqslant  \cdots \geqslant  \lambda_\ell)$ corresponds to the ideal
$$
\big(y^{\lambda_1}, x y^{\lambda_2}, \ldots, x^{\ell-1}y^{\lambda_\ell}, x^{\ell}\big),
$$
where $\ell(\lambda) = \ell$ is the length of $\lambda$. Hence we can index the points of the fixed locus $(S^{[m]})^T$ by collections of partitions 
$$
{\boldsymbol{\lambda}} = \big\{ \lambda^{(\sigma)} \big\}_{\sigma=1}^{e(S)}
$$
of total size 
$$
\sum_{\sigma=1}^{e(S)} |\lambda^{(\sigma)}| = \sum_{\sigma=1}^{e(S)} \sum_{i=1}^{\ell(\lambda^{(\sigma)})} \lambda_i^{(\sigma)} = m.
$$
Denote the closed subscheme corresponding to ${\boldsymbol{\lambda}}$ by $Z_{{\boldsymbol{\lambda}}}$. 

In order to calculate integrals such as the one in Theorem \ref{mainthm} by Atiyah-Bott localization, we need to consider Chern classes of
\begin{align}
\begin{split} \label{restrictions}
\mathcal{L}^{[m]}|_{Z_{{\boldsymbol{\lambda}}}} &= H^0(\mathcal{L}|_{Z_{{\boldsymbol{\lambda}}}}) \in K_0^{T}(\mathrm{pt}) = \mathbb{Z}[t_1^{\pm1}, t_2^{\pm1}], \\
\Big(\dR \Gamma(\mathcal{L}) \otimes \oO -\dR\hH om_{\pi}(\iI,\iI \boxtimes \mathcal{L}) \Big) \Big|_{Z_{{\boldsymbol{\lambda}}}} &= \dR\Gamma(\mathcal{L}|_{Z_{{\boldsymbol{\lambda}}}}) -\RHom_S(I_{Z_{{\boldsymbol{\lambda}}}}, I_{Z_{{\boldsymbol{\lambda}}}} \otimes \mathcal{L}) \in K_0^{T}(\mathrm{pt}),
\end{split}
\end{align}
where $t_1, t_2$ are the equivariant parameters of $T$. 

Suppose $Z_{{\boldsymbol{\lambda}}}$ is a 0-dimensional $T$-equivariant subscheme supported entirely on a maximal $T$-invariant affine open subset $U_\sigma$ and set $\lambda := \lambda^{(\sigma)}$. Suppose we choose coordinates such that $U_\sigma = \Spec \mathbb{C}[x,y]$ and the torus action (on coordinate functions) is given by $(t_1,t_2) \cdot (x,y) = (t_1x,t_2y)$. Denote the character corresponding to $\mathcal{L}|_{U_\sigma}$ by $\chi(t_1,t_2)$. Then
\begin{align}
\begin{split} \label{Kclass 1}
H^0(\mathcal{L}|_{Z_{{\boldsymbol{\lambda}}}}) &= \chi(t_1,t_2) \cdot Z_\lambda, \\
\mathrm{where} \quad Z_\lambda&:= \sum_{i = 1}^{\ell(\lambda)} \sum_{j=1}^{\lambda_i} t_1^{i-1} t_2^{j-1}.
\end{split}
\end{align}
Now suppose $W_{{\boldsymbol{\mu}}}$ is a second 0-dimensional $T$-equivariant subscheme supported entirely on $U_\sigma$ and write $\mu := \mu^{(\sigma)}$.
The following formula can be deduced from a well-known calculation using \v{C}ech cohomology (e.g.~see \cite[Prop.~4.1]{GK}): 
\begin{align}
\begin{split} \label{Kclass 2}
\RHom_S(\oO_{W_{\boldsymbol{\mu}}}, \oO_{Z_{\boldsymbol{\lambda}}} \otimes \mathcal{L}) &= \chi(t_1,t_2) \, W_\mu^* Z_\lambda \frac{(1-t_1) (1- t_2)}{t_1t_2} \in K_0^T(\mathrm{pt}), \\
\mathrm{where} \quad W_\mu &:= \sum_{i = 1}^{\ell(\mu)} \sum_{j=1}^{\mu_i} t_1^{i-1} t_2^{j-1}.
\end{split}
\end{align}
Here $(\cdot)^*$ is the involution defined by dualizing. 

For arbitrary $Z_{{\boldsymbol{\lambda}}}$, the $K$-group classes of \eqref{restrictions} can be determined from \eqref{Kclass 1} and \eqref{Kclass 2} by using the following equalities in $K$-theory
\begin{align*}
\oO_{Z_{{\boldsymbol{\lambda}}}} &= \sum_{\sigma=1}^{e(S)} \oO_{Z_{\lambda^{(\sigma)}}}, \\ 
I_{\oO_{Z_{{\boldsymbol{\lambda}}}}} &= \oO_S - \oO_{Z_{{\boldsymbol{\lambda}}}},
\end{align*}
where $Z_{\lambda^{(\sigma)}}$ denotes the 0-dimensional closed subscheme supported on $U_\sigma$ determined by $\lambda^{(\sigma)}$.

Consider Theorem \ref{mainthm} for the examples of $(S,L_1,L_2)$ listed in Proposition \ref{stable pair on del pezzo} and Remark \ref{CKK for nef cases}. In each case, we calculated the invariant $P_{n,\beta}([\pt])$ by first integrating out the linear system $\mathbb{P}^{\chi(\beta)-1}$. This amounts to expanding the integrand in powers of $h = c_1(\oO(1))$ and taking the coefficient of $h^{\chi(\beta)-1}$. This gives a polynomial expression $P$ in Chern classes of complexes of the form \eqref{taut complexes}. The integral $\int_{S^{[m]}} P$ is then calculated by Atiyah-Bott localization as described. The resulting stable pair invariants are tabulated in Appendix \ref{tables}. 

With the numbers of Appendix \ref{tables}, we are able to do various new checks of the Cao-Maulik-Toda conjectures (Conjectures \ref{conj:GW/GV g=0} and \ref{conj:GW/GV g=1}). Combining our tables in Appendix \ref{tables} with the tables for $n_{0,\beta}([\pt]), n_{1,\beta}$ in \cite[Sect.~3]{KP} gives Corollary \ref{main cor 1} of the introduction.

Bousseau-Brini-van Garrel \cite{BBG} determined all the genus zero Gromov-Witten (and therefore Gopakumar-Vafa type) invariants of $X = \mathrm{Tot}_S(L_1 \oplus L_2)$ with $(S,L_1,L_2)$ as in Remark \ref{CKK for nef cases} (as well as for  other cases). Note that the method of Bousseau-Brini-van Garrel does not produce genus one Gopakumar-Vafa type invariants, so we can only do verifications of Conjecture \ref{conj:GW/GV g=0} in these cases. Combining the tables in Appendix \ref{tables} with the values for genus zero Gopakumar-Vafa type invariants provided to us by Bousseau-Brini-van Garrel gives Corollary \ref{main cor 2} of the introduction.

Recall that for $(S,L_1,L_2)$ with $L_1 \otimes L_2 \cong K_S$, $L_1^{-1}$, $L_2^{-1}$ non-trivial and nef, and $S$ minimal and toric, we classified all values of $n \geqslant 0$ for which $P_n(S,\beta) \cong P_n(X,\beta)$, and $P_n(S,\beta) \neq \varnothing$ (Remark \ref{CKK for nef cases}). In these cases we therefore calculated \emph{all} stable pair invariants $P_{n,\beta}([\pt])$.

\appendix

\section{Tables} \label{tables}

\subsection{Local surfaces}\label{local surfa list}

In this section, we list the stable pair invariants $P_{n,\beta}([\mathrm{pt}])$ of $X = \mathrm{Tot}_S(L_1 \oplus L_2)$ for the cases mentioned in Proposition \ref{stable pair on del pezzo} and Remark \ref{CKK for nef cases}, and for a few additional cases. We use the following conventions and notation:
\begin{itemize}
\item $P_{0,0}:=1$ and $P_{n,0}([\pt]) = 0$ for all $n > 0$.
\item Entries decorated with $\star$ were defined by a virtual  localization formula on the fixed locus and have been calculated using the vertex formalism as discussed in \cite{CK2, CKM}. In these cases $P_n(X,\beta) \setminus P_n(S,\beta) \neq \varnothing$. See Remark \ref{virtloc} for a comparison to the globally defined invariants. All other entries have been computed using Theorem \ref{mainthm},
\item Zeroes decorated with $\dagger$ have non-empty underlying moduli space $P_n(X,\beta)$. In this sense, they are ``non-trivial'' zeroes.
\end{itemize}
\noindent For $(S,L_1,L_2) = (\mathbb{P}^2,\oO(-1),\oO(-2))$, we calculated the following values for $P_{0,d} :=P_{0,d[H]}$ and $P_{n,d}([\pt]):=P_{n,d[H]}([\pt])$ (for $n >0$). 
\begin{center}
\begin{table}[H]
\begin{tabular} {|c||c|c|c|c|c|}
\hline
$d \setminus n$ &0 &1&2&3&4\\
\hline\hline
1&  0& $-1$ & $0^{\dagger}$ & $0^{\dagger}$ & $0^{\dagger}$\\
\hline
2& 0 & 1 & 1& $0^{\star, \dagger}$ & \\
\hline
3&  $-1$ & $-1$ & $-2$ & & \\
\hline
4&  2&3  & & & \\
\hline
\end{tabular}
\hfill $$ $$
$$
P_{n,1}([\pt])=0^{\dagger}, \,\,\,  \forall \,\, n  \geqslant 2.
$$
\end{table}
\end{center}

\vspace{-0.5cm}

\noindent For $(S,L_1,L_2) = (\mathbb{P}^1 \times \mathbb{P}^1,\oO(-1,-1),\oO(-1,-1))$ and $P_{0,(d_1,d_2)}:=P_{0,d_1[H_1] + d_2[H_2]}$ and $P_{n,(d_1,d_2)}([\pt]):=P_{n,d_1[H_1] + d_2[H_2]}([\pt])$ (for $n >0$), where $H_1$, $H_2$ are defined in Example \ref{P1xP1 cpt/noncpt}, we have  
\begin{center}
\begin{table}[H]
\begin{tabular} {|c||c|c|c|c|c|}
\hline
$P_{0,(d_1,d_2)}$ &0 &1&2&3&4\\
\hline\hline
0&1&0&0&0&0\\
\hline
1&0&0&0&0&0\\
\hline
2&0&0&1&2&5\\
\hline
3&0&0&2&10& \\
\hline
4&0&0&5& & \\
\hline
\end{tabular}
\quad
\begin{tabular} {|c||c|c|c|c|c|}
\hline
$P_{1,(d_1,d_2)}([\pt])$ &0 &1&2&3&4\\
\hline\hline
0&0&1&0&0&0\\
\hline
1&1&1&1&1&1\\
\hline
2&0&1&2&5& \\
\hline
3&0&1&5& & \\
\hline
4&0&1& & & \\
\hline
\end{tabular}
\hfill $$ $$
\begin{tabular} {|c||c|c|c|c|c|}
\hline
$P_{2,(d_1,d_2)}([\pt])$ &0 &1&2&3&4\\
\hline\hline
0&0&$0^{\dagger}$&1&0&0\\
\hline
1&$0^{\dagger}$&2&2&2&2\\
\hline
2&1&2&5& & \\
\hline
3&0&2& & & \\
\hline
4&0&2& & & \\
\hline
\end{tabular}
\quad
\begin{tabular} {|c||c|c|c|c|c|}
\hline
$P_{3,(d_1,d_2)}([\pt])$ &0 &1&2&3&4\\
\hline\hline
0&0&$0^{\dagger}$&$0^{\star,\dagger}$&1&0\\
\hline
1&$0^{\dagger}$&$0^{\dagger}$&$3^{\star}$& & \\
\hline
2&$0^{\star,\dagger}$&$3^{\star}$& & & \\
\hline
3&1& & & & \\
\hline
4&0& & & & \\
\hline
\end{tabular}
\end{table}
\end{center}
\begin{align*}
P_{0,(1,d)} &= P_{0,(d,1)}=0, \,\,\,  \forall \,\,d \geqslant 0 \\
P_{1,(1,d)}([\pt]) &= P_{1,(d,1)}([\pt])=1, \,\,\,  \forall \,\,d \geqslant 0 \\
P_{2,(1,d)}([\pt]) &= P_{2,(d,1)}([\pt])=2, \,\,\,  \forall \,\, d \geqslant 1 \\
P_{n,(0,d)}([\pt]) &= P_{n,(d,0)}([\pt]) =\delta_{n, d}, \,\,\,  \forall \,\, 0 \leqslant n \leqslant d \\
P_{n,(1,1)}([\pt])&=0^{\dagger}, \,\,\,  \forall \,\,n \geqslant 3,\\
P_{n,(0,1)}([\pt])&=P_{n,(1,0)}([\pt])=0^{\dagger}, \,\,\,  \forall \,\,n \geqslant 2.
\end{align*}

\noindent For $(S,L_1,L_2) = (\mathbb{P}^1 \times \mathbb{P}^1,\oO(-1,0),\oO(-1,-2))$, we have:
\begin{align*}
&P_{0, (2,2)}=1, \,\,\,  P_{0, (2,3)}=2, \,\,\,  P_{0, (2,4)}=5, \,\,\, P_{0, (3,2)}=2, \\
&P_{1, (2,2)}([\pt])=2, \,\,\, P_{1, (2,3)}([\pt])=5, \\
&P_{1, (d,1)}([\pt]) = P_{1, (1,d)}([\pt]) = 1, \,\,\, \forall \,\, d \geqslant 1 \\
&P_{2, (1,d)}([\pt]) = 2, \,\,\, \forall \,\, d \geqslant 2, \\
&P_{n,(0,1)}([\pt])=0^{\dagger}, \,\,\,  \forall \,\,n \geqslant 2\\
&P_{n,(0,n)}([\pt]) = 1, \,\,\, \forall \,\,n \geqslant 1.
\end{align*}

\noindent Recall the notation for Hirzebruch surfaces from the introduction. Consider $(S,L_1,L_2) = (\mathbb{F}_1, \oO(-1,-1 ), \oO(-1,-2))$. We write $P_{0,(d_1,d_2)}:=P_{0,d_1[B] + d_2 [F]}$ and, for $n > 0$,  $P_{n,(d_1,d_2)}([\pt]):=P_{n,d_1[B] + d_2 [F]}([\pt])$. In the tables below, the rows are for $d_1$ and the columns for $d_2$.
{\footnotesize{
\begin{center}
\begin{table}[H]
\begin{tabular}{|c||c|c|c|c|c|c|c|c|}
\hline
$P_{0,(d_1,d_2)}$ &0 &1&2&3&4&5&6&7\\
\hline\hline
0& 1 &0&0&0&0&0&0&0\\
\hline
1&0&0&0&0&0&0&0&0\\
\hline
2&0&0&0& $1$ &2&5& & \\
\hline
3&0&0&0& $-1$ & & & & \\
\hline
\end{tabular}
\quad
\begin{tabular}{|c||c|c|c|c|c|c|c|c|}
\hline
$P_{1,(d_1,d_2)}([\pt])$ &0 &1&2&3&4&5&6&7\\
\hline\hline
0& 0 & $1$ &0&0&0&0&0&0\\
\hline
1& & $-1$ & $-1$ & $-1$ & $-1$ & $-1$ & $-1$ & $-1$\\
\hline
2&0&0&1&2&5 & & & \\
\hline
3&0&0&0& & & & & \\
\hline
\end{tabular}
\end{table}
\end{center}
}}

\vspace{-0.8cm}

\begin{align*}
&P_{1,(1,d)}([\pt])=-1, \quad \forall \,\, d \geqslant 1 \\
&P_{2,(1,d)}([\pt])=- 2, \quad \forall \,\,d \geqslant 2 \\
&P_{3,(0,2)}([\pt])=0^{\star, \dagger}, \\
&P_{n,(0,1)}([\pt])=0^{\dagger}, \,\,\,  \forall \,\,n \geqslant 2\\
&P_{n,(0,n)}([\pt]) = 1, \,\,\, \forall \,\,n \geqslant 1.
\end{align*}

\begin{rmk}
Denoting the exceptional curve of $\mathbb{F}_1$ by $B$, we have $N_{B/X} = \oO(-1)\oplus\oO \oplus \oO(-1)$, which has sections in the direction of $L_1$. Therefore $P_1(X,[B])$ is non-proper, which explains the gap in the table for $P_{1,(1,0)}([\pt])$.  
\end{rmk}
	
\noindent For $(S,L_1,L_2) = (\mathbb{F}_1, \oO(0,-1), \oO(-2,-2))$, we have
\begin{center}
\begin{table}[H]
\begin{tabular}{|c||c|c|c|c|c|}
\hline
$P_{0,(d_1,d_2)}$ &0 &1&2&3&4\\
\hline\hline
0&$1$ &0 &0&0&0\\
\hline
1&0&0&0&0&0\\
\hline
2&0&0&0&1&2\\
\hline
3&0&0&0&$-1$& \\
\hline
\end{tabular}
\quad
\begin{tabular}{|c||c|c|c|c|c|}
\hline
$P_{1,(d_1,d_2)}([\pt])$ &0 &1&2&3&4\\
\hline\hline
0&$0 $ &  &0&0&0\\
\hline
1 & &$-1$&$-1$&$-1$&$-1$\\
\hline
2&0&0&$1$&2& \\
\hline
3&0&0&0& & \\
\hline
\end{tabular}
\hfill $$ $$
\begin{align*}
P_{1,(1,d)}([\pt])&=-1, \quad \forall \,\, d \geqslant 1 
\end{align*}
\end{table}
\end{center}

\vspace{-0.8cm}
	
\noindent For $(S,L_1,L_2) = (\mathbb{F}_2,\oO(-1,-2), \oO(-1,-2))$, we have
\begin{center}
\begin{table}[H]
\begin{tabular}{|c||c|c|c|c|c|c|c|c|}
\hline
$P_{0,(d_1,d_2)}$ &0 &1&2&3&4&5&6&7\\
\hline\hline
0&$1 $&0&0&0&0&0&0&0\\
\hline
1&0&0&0&0&0&0&0&0\\
\hline
2&0&0&0&0&1&2&5& \\
\hline
\end{tabular}
\quad
\begin{tabular}{|c||c|c|c|c|c|c|c|c|}
\hline
$P_{1,(d_1,d_2)}([\pt])$ &0 &1&2&3&4&5&6&7\\
\hline\hline
0& 0 &1&0&0&0&0&0&0\\
\hline
1& &1& $1$ &1&1&1&1&1\\
\hline
2&0&0&0&1&2&5& & \\
\hline
\end{tabular}
\end{table}
\end{center}

\begin{align*}
& P_{1,(1,d)}([\pt])=1, \quad \forall \,\, d \geqslant 1 \\
&P_{2,(1,d)}([\pt])=2, \quad \forall \,\,d \geqslant 2 \\
&P_{n,(0,1)}([\pt])=0^{\dagger}, \,\,\,  \forall \,\,n \geqslant 2\\
&P_{n,(0,n)}([\pt]) = 1,  \,\,\, \forall \,\,n \geqslant 1.
\end{align*}

 \subsection{Local $\mathbb{P}^3$}
 
Consider $X = \mathrm{Tot}_{\mathbb{P}^3}(K_{\mathbb{P}^3})$. Let $P_{0,d} := P_{0,d[\ell]}$ and $P_{n,d}([\ell]) := P_{n,d[\ell]}([\ell])$ (for $n>0$), where $[\ell] \in H_2(\mathbb{P}^3,\mathbb{Z}) \cong H_2(X,\mathbb{Z})$ denotes the class of a line $\ell \subseteq \mathbb{P}^3$ and we also write
$[\ell] \in H^4(X,\mathbb{Z})$ for the pull-back of its Poincar\'e dual from $\mathbb{P}^3$ to $X$. 
Obviously, $X = \mathrm{Tot}_{\mathbb{P}^3}(K_{\mathbb{P}^3})$ is not a local surface so Theorem \ref{mainthm} does not apply. All stable pair invariants in this section have been calculated using the vertex formalism of \cite{CK2, CKM} (this is stressed by decorating the invariants with $\star$). We determined the following values of $P_{0,d}$ and $P_{n,d}([\ell])$.
 \begin{center}
\begin{table}[H]
\begin{tabular} {|c||c|c|c|c|c|}
\hline
$d \setminus n$ &0 &1&2&3&4\\
\hline\hline
$1$ & $0^{\star}$ & $-20^{\star}$ & $0^{\star, \dagger}$ & $0^{\star, \dagger}$ & $0^{\star, \dagger}$ \\
\hline
$2$ & $0^{\star}$ & $-820^{\star}$ & $400^{\star}$ &&\\
\hline
$3$ & $11200^{\star}$ & $-68060^{\star}$ & &&\\
\hline
\end{tabular}
\hfill $$ $$
\begin{align*}
P_{n,1}([\ell])=0, \,\,\, \forall \,\, n \geqslant 2.
\end{align*}
\end{table}
\end{center}

\vspace{-0.8cm}

\begin{rmk} \label{virtloc}
For $X =\mathrm{Tot}_{\mathbb{P}^3}(K_{\mathbb{P}^3})$ and all the cases in this table, we have $P_n(\mathbb{P}^3,\beta) \cong P_n(X,\beta)$. This can be deduced from a filtration argument similar to Proposition \ref{stable pair on del pezzo} combined with the fact that all degree 2 Cohen-Macaulay curves $C$ on $\mathbb{P}^3$ satisfy $\chi(\oO_C) \geqslant 1$ (see also \cite[p.~20]{CT2}).
Therefore, the reasoning of Proposition \ref{compare virtual class} yields 
\begin{equation} \label{eqnapp}
[P_n(X,\beta)]^{\mathrm{vir}}=(-1)^{\beta \cdot c_1(\mathbb{P}^3) +n}\cdot [P_n(\mathbb{P}^3,\beta)]_{\mathrm{pair}}^{\mathrm{vir}},  
\end{equation}
where $[P_n(\mathbb{P}^3,\beta)]_{\mathrm{pair}}^{\mathrm{vir}}$ is the virtual class of the pair perfect obstruction theory \eqref{pairpot} on $\mathbb{P}^3$ (see also \cite[Lem.~3.1]{CMT2} in a similar setting). The sign in this formula is 
a preferred choice of orientation on $P_n(X,\beta)$ as for \eqref{intro choice of ori}. 
Now we are in the world of ``ordinary'' perfect obstruction theories and the torus action on $\mathbb{P}^3$ can be used to apply the Graber-Pandharipande virtual localization formula \cite{GP} to the right hand side of \eqref{eqnapp}. Similar to the argument for \cite[Thm.~A.1]{CK2}, it then follows that the invariants in this table (defined by localization on the fixed locus \cite{CK2, CKM}) are equal to the global invariants \eqref{glob inv}. This reasoning also works for the local surface case $(S,L_1,L_2) = (\mathbb{P}^2,\oO(-1),\oO(-2))$, $d=2$, $n=3$, because then all stable pairs are scheme theoretically supported in the threefold $Y = \mathrm{Tot}_S(L_1)$.
\end{rmk}

\end{document}